\documentclass[a4paper,UKenglish,cleveref, autoref]{lipics-v2019}

\title{\Buchi\ Objectives in Countable MDPs}
\author{Stefan Kiefer}{University of Oxford, UK}{}{}{Supported by a Royal Society University Research Fellowship.}
\author{Richard Mayr}{University of Edinburgh, UK}{}{}{Supported by EPSRC grant EP/M027651/1.}
\author{Mahsa Shirmohammadi}{CNRS \& IRIF, FR}{}{}{Supported by PEPS JCJC grant AAPS.}
\author{Patrick Totzke}{University of Liverpool, UK}{}{}{}
\authorrunning{S.~Kiefer, R.~Mayr, M.~Shirmohammadi, P.~Totzke}
\Copyright{Stefan Kiefer, Richard Mayr, Mahsa Shirmohammadi, Patrick Totzke}

\relatedversion{}

\supplement{}

\supplement{}

\funding{}

\acknowledgements{The authors thank anonymous reviewers for their helpful comments.}

\nolinenumbers 

\hideLIPIcs  

\EventEditors{Christel Baier, Ioannis Chatzigiannakis, Paola Flocchini, and Stefano Leonardi}
\EventNoEds{4}
\EventLongTitle{46th International Colloquium on Automata, Languages, and Programming (ICALP 2019)}
\EventShortTitle{ICALP 2019}
\EventAcronym{ICALP}
\EventYear{2019}
\EventDate{July 9--12, 2019}
\EventLocation{Patras, Greece}
\EventLogo{eatcs}
\SeriesVolume{132}
\ArticleNo{115}

\usepackage{amsmath,amsthm}
\usepackage{pdfrender}
\usepackage{epsfig,psfrag}
\usepackage{graphicx}
\usepackage{graphics}
\usepackage{latexsym}
\usepackage{amssymb}
\usepackage{amsmath}
\usepackage{stmaryrd}
\usepackage{ifthen}
\usepackage{array}
\usepackage{mathrsfs}
\usepackage{paralist}
\usepackage{multirow}

\usepackage{pgf}
\usepackage{tikz}
\usetikzlibrary{arrows,automata,backgrounds,decorations}
\usetikzlibrary{shapes.multipart,shapes.misc}
\usepgflibrary{decorations.pathreplacing}
\usetikzlibrary{arrows,calc,topaths}
\usetikzlibrary{automata,shapes.arrows, backgrounds, fadings,intersections, positioning}
\usetikzlibrary{shadows}
\tikzstyle{small}=[font=\footnotesize]
\tikzset{
    every picture/.style={>=stealth,auto,node distance=2cm},
}
\usepackage{pgfplots}
\usetikzlibrary{intersections, pgfplots.fillbetween}
\usepackage{tkz-euclide}
\usetkzobj{all}
\usetikzlibrary{shapes.misc}
\usetikzlibrary{petri}
\usetikzlibrary{shadows}
\usetikzlibrary{decorations.pathmorphing}
\usetikzlibrary{decorations.shapes}
\usetikzlibrary{shapes}
\usetikzlibrary{backgrounds}
\usetikzlibrary{fit}
\usetikzlibrary{arrows}
\usetikzlibrary{calc}
\usetikzlibrary{mindmap}
\usetikzlibrary{matrix}

 \usepackage{todonotes}

\usepackage{mathtools}
\usepackage{thmtools,thm-restate}
\usepackage{xcolor}

\usepackage{bbding}

\usepackage{hyperref} 
\hypersetup{colorlinks,linkcolor=blue,citecolor=blue,urlcolor=blue}
\crefname{equation}{}{}
\tikzset{every picture/.style={thick,>=angle 60}}
\tikzset{MDPrand/.style={draw,circle,minimum size=11*1.5,inner sep=0}}
\tikzset{MDPcont/.style={draw,rectangle,minimum size=9*1.5,inner sep=0,fill=yellow}}
\tikzset{MDPblue/.style={fill=blue!30}}
\tikzset{MDPgreen/.style={accepting,fill=green}}
\tikzset{MDPbrown/.style={fill=brown}}
\tikzset{MDPred/.style={fill=red!50}}
\tikzset{MDPddd/.style={MDPcont,draw=none,text height=2.5ex,text depth=.25ex,inner sep=2}}

\newcommand{\block}{
\node[MDPrand,MDPbrown] (1) at (-1,0) {B};
\node[MDPrand,MDPbrown] (2) at (0,0) {B};
\node[MDPrand] (3) at (1,0) {W};
\node[MDPcont] (4) at (-1.0,-1) {Y};
\node[MDPcont] (5) at (0.5,-1) {Y};
\draw[->] (1) edge node[pos=0.37,'] {$0.3$} (2);
\draw[->] (1) edge node[pos=0.3,inner sep=0.5] {$0.7$} (4);
\draw[->] (2) edge node[pos=0.37,'] {$0.3$} (3);
\draw[->] (2) edge node[pos=0.4,inner sep=0] {$0.7$} (5);
\draw[->] (4) edge                         (2);
\draw[->] (5) edge                         (3);
}

\newcommand{\gblock}[1]{
\def\g{#1}
\node[MDPrand,draw=none] (\g1) at (-1,0) {};
\node[MDPrand,draw=none] (\g2) at (0,0) {};
\node[MDPrand,draw=none] (\g3) at (1,0) {};
\node[MDPcont,draw=none] (\g4) at (-1.0,-1) {};
\node[MDPcont,draw=none] (\g5) at (0.5,-1) {};
}

\newcommand{\redblock}[2]{
\node[MDPrand,MDPred] (0) at (#1-1,-2) {R};
\draw[->] (0) edge[bend left =16] node[pos=0.45,inner sep=-0.5pt] {$\frac{#2^2}{#2^2+1}$} (1);
\draw[->] (0) edge[red,bend right=16] node[pos=0.45,',red,inner sep=-0.5pt] {$\frac{1}{#2^2+1}$} (1);
}

\newcommand{\Tzero}[2]{
\begin{scope}[shift={(#1,-3)}]
\node[MDPrand,MDPgreen] (1) at (-0/2,0) {G};
\end{scope}
}

\newcommand{\Tone}[2]{
\begin{scope}[shift={(#1,-3)}]
\gblock{1}
\Tzero{-1}{#2}
\redblock{-0/2}{#2}
\draw[->] (14) edge (0);
\draw[->,rounded corners=3mm] (1) -- +(0.75,0) -- (12);
\Tzero{0.5}{#2}
\redblock{+3/2}{#2}
\draw[->] (15) edge (0);
\draw[->,rounded corners=3mm] (1) -- +(0.80,0) -- (13);
\block{1}
\end{scope}
}

\newcommand{\Ttwo}[2]{
\begin{scope}[shift={(#1,-3)}]
\gblock{2}
\Tone{-1.5}{#2}
\redblock{-1.5}{#2}
\draw[->] (24) edge (0);
\draw[->] (3) edge[bend right=0] (22);
\Tone{1.5}{#2}
\redblock{1.5}{#2}
\draw[->] (25) edge (0);
\draw[->] (3) -- (23);
\block
\end{scope}
}

\newcommand{\Tthree}[2]{
\begin{scope}[shift={(#1,-3)}]
\gblock{3}
\Ttwo{-3}{#2}
\redblock{-3}{#2}
\draw[->] (34) edge (0);
\draw[->] (3) edge[bend right=20] (32);
\Ttwo{+3}{#2}
\redblock{+3}{#2}
\draw[->] (35) edge (0);
\draw[->] (3) -- (33);
\block
\end{scope}
}

\newcommand{\onebit}{
\node[MDPrand,MDPblue] (r) at (-2,-3+0) {L};
\draw[->] (-2.7,-3+0) -- (r);
\Tone{0}{1}
\draw[->] (r) edge[bend right=20] node['] {$\frac11$} (1);
\draw[->] (r) edge[bend left =20] node {$1 - \frac11$} (3);
\node[MDPrand,MDPblue] (r) at (3) {L};
\Ttwo{4.5}{2}
\draw[->] (r) edge[bend right=20] node['] {$\frac12$} (1);
\draw[->] (r) edge[bend left =20] node {$1 - \frac12$} (3);
\node[MDPrand,MDPblue] (r) at (3) {L};
\Tthree{13.5}{3}
\draw[->] (r) edge[bend right=20] node['] {$\frac13$} (1);
\draw[->] (r) edge[bend left =20] node {$1 - \frac13$} (3);
\node[MDPrand,MDPblue] (r) at (3) {L};

\coordinate (r1) at (18.2,-3+0.9);
\draw[->] (r) edge[bend left =5] node[pos=0.8] {$1 - \frac14$} (r1);
\node[circle,fill,minimum size=3pt] at (18.5,-3+0.9) {};
\node[circle,fill,minimum size=3pt] at (18.8,-3+0.9) {};
\node[circle,fill,minimum size=3pt] at (19.1,-3+0.9) {};
\coordinate (r1) at (18.2,-3-0.9);
\draw[->] (r) edge[bend right=5] node[pos=0.7,'] {$\frac14$} (r1);
\node[circle,fill,minimum size=3pt] at (18.5,-3-0.9) {};
\node[circle,fill,minimum size=3pt] at (18.8,-3-0.9) {};
\node[circle,fill,minimum size=3pt] at (19.1,-3-0.9) {};
}

\declaretheorem[name=Theorem,style=plain]{ourtheorem}

\declaretheorem[name=Lemma,Refname={Lemma,Lemmas},sibling=ourtheorem]{ourlemma}

\declaretheorem[name=Example,style=definition,qed=\qedsymbol,sibling=ourtheorem]{ourexample}

\newcommand{\+}[1]{\mathbb{#1}}

\newcommand{\N}{\+{N}}
\newcommand{\R}{\+{R}}
\newcommand{\x}{\times}

\usepackage{wasysym} 
\newcommand{\rsymbol}{\ocircle}
\newcommand{\zsymbol}{\Box}
\newcommand{\zstates}{\states_\zsymbol}
\newcommand{\rstates}{\states_\rsymbol}
\newcommand{\reachset}{T}

\newcommand{\eqby}[2][=]{\stackrel{\text{{\tiny{#2}}}}{#1}}
\newcommand{\eqdef}{\eqby{def}}

\newcommand{\defeq}{\stackrel{\text{def}}{=}}

\newcommand{\eps}{\varepsilon}
\newcommand{\step}[2][]{\xrightarrow[#1]{#2}}

\newcommand{\problemx}[3]{
\par\noindent\underline{\sc#1}\par\nobreak\vskip.2\baselineskip
\begingroup\clubpenalty10000\widowpenalty10000
\setbox0\hbox{\bf INPUT:\ }\setbox1\hbox{\bf QUESTION:\ }
\dimen0=\wd0\ifnum\wd1>\dimen0\dimen0=\wd1\fi
\vskip-\parskip\noindent
\hbox to\dimen0{\box0\hfil}\hangindent\dimen0\hangafter1\ignorespaces#2\par
\vskip-\parskip\noindent
\hbox to\dimen0{\box1\hfil}\hangindent\dimen0\hangafter1\ignorespaces#3\par
\endgroup}

\newcommand{\dist}{\mathcal{D}}

\newcommand{\Parity}[1]{\mathtt{Parity}(#1)}

\newcommand{\hide}[1]{}

\newcommand{\lrc}[1]{(#1)}

\newcommand{\ignore}[1]{}

\newcommand{\tuple}[1]{\lrc{#1}}

\newcommand{\mdp}{{\mathcal M}}

\newcommand{\mdptuple}{\tuple{\states,\zstates,\rstates,\transition,\probp}}

\newcommand{\initstates}{I}
\newcommand{\states}{S}

\newcommand{\state}{s}

\newcommand{\transition}{{\longrightarrow}}

\newcommand{\probp}{P}

\newcommand{\smallparg}[1]{{\bf #1}}

\newcommand{\complementof}[1]{\overline{#1}}

\newcommand{\play}{\rho}

\newcommand{\partialplay}{\rho}

\newcommand{\zstrat}{\sigma}

\newcommand{\zstratset}{\Sigma}

\newcommand{\memory}{{\sf M}}
\newcommand{\memconf}{{\sf m}}

\newcommand{\expectval}{{\mathcal E}}
\newcommand{\probm}{{\mathcal P}}

\newcommand{\formula}{{\varphi}}

\newcommand{\valueof}[2]{{\mathtt{val}_{#1}(#2)}}
\newcommand{\optval}{v^*}

\mathchardef\mhyphen="2D 

\newcommand{\bubble}[2]{{\sf bubble}_{#1}(#2)}
\newcommand{\firstinset}[1]{{\sf firstin}(#1)}

\newcommand{\M}{\mathcal{M}}

\newcommand{\Buchi}{{B\"uchi}}
\newcommand{\reset}{\mbox{\upshape\texttt{B\"uchi}}}
\newcommand{\setf}{\mbox{$\mathtt{Goal}$}}
\newcommand{\setfb}[2]{\setf^{{#1}}(#2)}
\begin{document}
\maketitle

\begin{abstract}
\noindent
We study countably infinite Markov decision processes with \Buchi\
objectives, which ask to visit a given subset $F$ of states infinitely often.
A question left open by T.P.~Hill in 1979~\cite{Hill:79} is
whether there always exist $\eps$-optimal Markov strategies, i.e., strategies
that base decisions only on the current state and the number of steps taken so far.
We provide a negative answer to this question by constructing a non-trivial
counterexample. On the other hand, we show that Markov strategies with only
$1$ bit of extra memory are sufficient.
\end{abstract}

\newpage
\section{Introduction}\label{sec:intro}
{\bf\noindent Background.}
Markov decision processes (MDPs) are a standard model for dynamic systems that
exhibit both stochastic and controlled behavior \cite{Puterman:book}.
MDPs play a prominent role in numerous domains, including artificial intelligence and machine learning~\cite{sutton2018reinforcement,sigaud2013markov}, control theory~\cite{blondel2000survey,NIPS2004_2569}, operations research and finance~\cite{bauerle2011finance,schal2002markov}, and formal verification~\cite{ModCheckHB18,ModCheckPrinciples08}.
In an MDP, the system starts in the initial state and makes a sequence of transitions
between states.
Depending on the type of the current state, either the controller gets to
choose an enabled transition (or a distribution over transitions), or the next
transition is chosen randomly according to a defined distribution.
By fixing a strategy for the controller, one obtains a probability space
of runs of the MDP. The goal of the controller is to optimize the expected value of
some objective function on the runs.

The type of strategy needed for an optimal (resp.\ $\eps$-optimal)
strategy for some objective is also called the \emph{strategy complexity} of
the objective.
There are different types of strategies, depending on whether one can take
the whole history of the run into account (history-dependent; (H)),
or whether one is limited to a finite amount of memory (finite memory; (F))
or whether decisions are based only on the current state (memoryless; (M)).
Moreover, the strategy type depends on whether the controller can
randomize (R) or is limited to deterministic choices (D).
The simplest type MD refers to memoryless deterministic strategies.
\emph{Markov strategies} are strategies that base their decisions only on
the current state and the number of steps in the history or the run.
Thus they do use infinite memory, but only in a very restricted form
by maintaining an unbounded step-counter.
For finite MDPs, there exist optimal MD-strategies for many (but not all)
objectives \cite{CAH:QEST2004,chatterjee2012survey,Chatterjee:2004:QSP:982792.982808,Puterman:book},
but the picture is more complex for countably infinite
MDPs \cite{KMSW2017,Ornstein:AMS1969,Puterman:book}.

We study here so-called \emph{Goal} objectives
defined via a subset of goal states $F$:
In the basic Goal objective
(also called the \emph{Reachability} objective) one simply wants to reach the set $F$.
In the \emph{\Buchi} objective one wants to visit the set $F$ infinitely often.
For finite MDPs there exist optimal MD-strategies for both these objectives \cite{chatterjee2012survey,Puterman:book}.
For countably infinite MDPs, optimal strategies (where they exist) and
$\eps$-optimal strategies for Reachability can be chosen
MD \cite{Ornstein:AMS1969,Puterman:book}.
Similarly, optimal strategies for \Buchi\ (where they exist) can be chosen MD
\cite{KMSW2017}. However, $\eps$-optimal strategies for
\Buchi\ require infinite memory (cannot be chosen FR); cf.\ \cite{KMSW2017,Krcal:Thesis:2009}.

 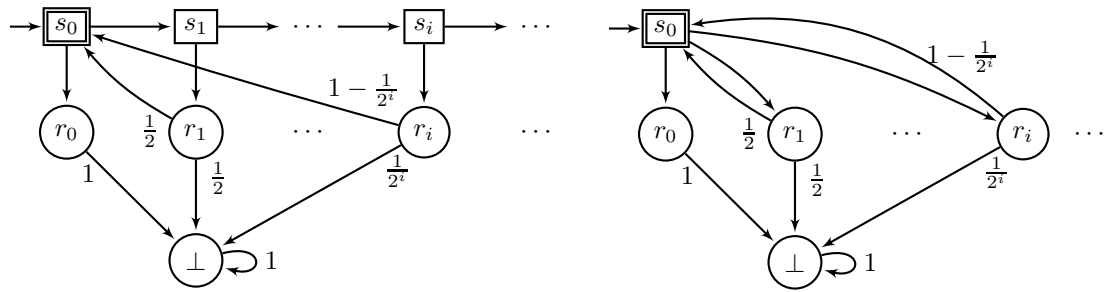
\begin{figure}[h]
    \begin{minipage}{0.45\textwidth}
        \begin{center}				

\centering
\begin{tikzpicture}[>=latex',shorten >=1pt,node distance=1.9cm,on grid,auto,
roundnode/.style={circle, draw,minimum size=1.5mm},
squarenode/.style={rectangle, draw,minimum size=2mm},
diamonddnode/.style={diamond, draw,minimum size=2mm}]

\node [squarenode,double,initial,initial text={}] (s0) at(0,0) [draw]{$s_0$};
\node [squarenode] (s1) [draw,right=1.7cm of s0]{$s_1$};
\node [squarenode] (s3) [draw=none,right=1.5cm of s1]{$\cdots$};
\node [squarenode] (s4) [draw,right=1.5cm of s3]{$s_i$};
\node [squarenode] (s5) [draw=none,right=1.5cm of s4]{$\cdots$};

\node[roundnode] (r0)  [below=1.4cm of s0] {$r_0$};
\node[roundnode] (r1)  [below=1.4cm of s1] {$r_1$};
\node[roundnode] (r3)  [draw=none,below=1.4 of s3] {$\cdots$};
\node[roundnode] (r4)  [below=1.4cm of s4] {$r_i$};
\node[roundnode] (r5)  [draw=none,below=1.4cm of s5] {$\cdots$};

\node [roundnode] (t)  [below=1.7cmof r1] {$\bot$};

\path[->] (s0) edge (s1);
\path[->] (s1) edge (s3);
\path[->] (s3) edge (s4);
\path[->] (s4) edge (s5);

\path[->] (s0) edge (r0);
\path[->] (s1) edge (r1);
\path[->] (s4) edge (r4);

\path[->] (r0) edge node [near start,left=.05cm] {$1$} (t);
\path[->] (r1) edge node [near start,right=.05cm] {$\frac{1}{2}$} (t);
\path[->] (r4) edge node [near start,right=.25cm] {$\frac{1}{2^{i}}$} (t);

\path[->] (r1) edge [bend left=10] node[pos=0.2,below] {$\frac{1}{2}$} (s0);
\path[->] (r4) edge [bend left=1] node [pos=0.1,above] {$1-\frac{1}{2^{i}}$} (s0);

\path[->] (t) edge [loop right] node[pos=0.5,right] {$1$}  ();	
\end{tikzpicture}
				\end{center}
					\textbf{a)}~Finitely branching, but infinitely many controlled states.
    \end{minipage}%
    \hspace{15mm}
    \begin{minipage}{0.45\textwidth}
        \begin{center}

\centering
\begin{tikzpicture}[>=latex',shorten >=1pt,node distance=1.9cm,on grid,auto,
roundnode/.style={circle, draw,minimum size=1.5mm},
squarenode/.style={rectangle, draw,minimum size=2mm},
diamonddnode/.style={diamond, draw,minimum size=2mm}]

\node [squarenode,double,initial,initial text={}] (s0) at(0,0) [draw]{$s_0$};

\node[roundnode] (r0)  [below=1.4cm of s0] {$r_0$};
\node[roundnode] (r1)  [below=1.4cm of s1] {$r_1$};
\node[roundnode] (r3)  [draw=none,below=1.4 of s3] {$\cdots$};
\node[roundnode] (r4)  [below=1.4cm of s4] {$r_i$};
\node[roundnode] (r5)  [draw=none,right=0.9cm of r4] {$\cdots$};

\node [roundnode] (t)  [below=1.7cmof r1] {$\bot$};


\path[->] (s0) edge (r0);
\path[->] (s0) edge [bend left=10] node[pos=0.2,below] {} (r1);
\path[->] (s0) edge [bend left=10] node[pos=0.2,below] {} (r4);

\path[->] (r0) edge node [near start,left=.05cm] {$1$} (t);
\path[->] (r1) edge node [near start,right=.05cm] {$\frac{1}{2}$} (t);
\path[->] (r4) edge node [near start,right=.25cm] {$\frac{1}{2^{i}}$} (t);

\path[->] (r1) edge [bend left=10] node[pos=0.2,below] {$\frac{1}{2}$} (s0);
\path[->] (r4) edge [bend right=25] node [pos=0.3,right] {$1-\frac{1}{2^{i}}$} (s0);

\path[->] (t) edge [loop right] node[pos=0.5,right] {$1$}  ();	
\end{tikzpicture}
				\end{center}
				\textbf{b)}~Infinitely branching, but just one controlled state.
    \end{minipage}
\caption{Two MDPs where $\eps$-optimal strategies for
\Buchi\ require infinite memory. Let $F=\{s_0\}$ be the set of goal states.
Here and throughout the paper we indicate goal states by double borders, and controlled states as rectangles.
}\label{fig:infmem-Buchi}
\end{figure}

\begin{ourexample} \label{ex-infmem-Buchi}
Consider the MDPs in \autoref{fig:infmem-Buchi}. 
Every finite memory (FR) strategy will
only attain probability $0$ for \Buchi\ in these examples \cite{KMSW2017}.
However, there exists an $\eps$-optimal Markov strategy for every $\eps>0$:
At the $i$-th time that state $\state_0$ is visited, pick the
successor state $r_{i+k}$ where $k$ is some sufficiently large number
depending on $\eps$, e.g., $k=\lceil\log_2(1/\eps)\rceil$.
For example (b) this can easily be done with a step-counter since $s_0$ is
visited for the $i$-th time in step $2(i-1)$ unless the system has reached the
state $\bot$.
For example (a), under this strategy, state $s_0$ is visited for the $i$-th time in step
$\sum_{j=1}^{i-1} (k+j+1)$ unless the system has reached the
state $\bot$.
\end{ourexample}

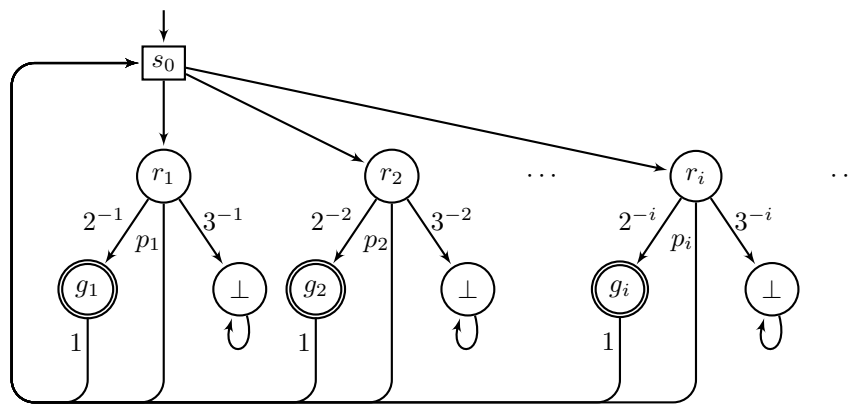
\begin{figure}[h]
\centering
\begin{tikzpicture}[>=latex',shorten >=1pt,node distance=1.9cm,on grid,auto,
roundnode/.style={circle, draw,minimum size=1.5mm},
squarenode/.style={rectangle, draw,minimum size=2mm},
diamonddnode/.style={diamond, draw,minimum size=2mm}]
\node [squarenode] (s0) at(0,0) {$s_0$};
\path[->,draw] (0,0.7) -- (s0);
\node [roundnode] (r1) at (0,-1.5) {$r_1$};
\node [roundnode] (r2) at (3,-1.5) {$r_2$};
\node [roundnode,draw=none] at (5,-1.5) (r3){$\cdots$};
\node [roundnode] (ri) at (7,-1.5) {$r_i$};
\node [roundnode,draw=none] at (9,-1.5) {$\cdots$};
\node [roundnode,double] (g1) at (-1,-3) {$g_1$};
\node [roundnode,double] (g2) at (2,-3) {$g_2$};
\node [roundnode,double] (gi) at (6,-3) {$g_i$};
\node [roundnode] (b1) at (1,-3) {$\bot$};
\node [roundnode] (b2) at (4,-3) {$\bot$};
\node [roundnode] (bi) at (8,-3) {$\bot$};
\path[->] (b1) edge [loop below]  ();
\path[->] (b2) edge [loop below]  ();
\path[->] (bi) edge [loop below]  ();
\path[->,draw] (s0) -- (r1);
\path[->,draw] (s0) -- (r2);
\path[->,draw] (s0) -- (ri);
\path[->] (r1) edge node[left,pos=0.3] {$2^{-1}$} (g1);
\path[->] (r2) edge node[left,pos=0.3] {$2^{-2}$} (g2);
\path[->] (ri) edge node[left,pos=0.3] {$2^{-i}$} (gi);
\path[->] (r1) edge node[right,pos=0.3] {$3^{-1}$} (b1);
\path[->] (r2) edge node[right,pos=0.3] {$3^{-2}$} (b2);
\path[->] (ri) edge node[right,pos=0.3] {$3^{-i}$} (bi);
\path[draw,->,rounded corners=3mm] (g1) -- node[inner sep=0,left=0.05,pos=0.3] {$1$} (-1,-4.5) -- (-2,-4.5) -- (-2,0) -- (s0); \path[draw,->,rounded corners=3mm] (g2) -- node[inner sep=0,left=0.05,pos=0.3] {$1$} (2,-4.5) -- (-2,-4.5) -- (-2,0) -- (s0); \path[draw,->,rounded corners=3mm] (gi) -- node[inner sep=0,left=0.05,pos=0.3] {$1$} (6,-4.5) -- (-2,-4.5) -- (-2,0) -- (s0); \path[draw,->,rounded corners=3mm] (r1) -- node[inner sep=0,left=0.02,pos=0.2] {$p_1$} (0,-4.5) -- (-2,-4.5) -- (-2,0) -- (s0); \path[draw,->,rounded corners=3mm] (r2) -- node[inner sep=0,left=0.02,pos=0.2] {$p_2$} (3,-4.5) -- (-2,-4.5) -- (-2,0) -- (s0);
\path[draw,->,rounded corners=3mm] (ri) -- node[inner sep=0,left=0.02,pos=0.2] {$p_i$} (7,-4.5) -- (-2,-4.5) -- (-2,0) -- (s0);
\end{tikzpicture}
\caption{An MDP where $\eps$-optimal strategies for \Buchi\ require infinite memory.
The transition probability $p_i$ stands for $1 - 2^{-i} - 3^{-i}$.
The state $s_0$ is the only controlled state. 
}
\label{fig:hill99}
\end{figure}

\begin{ourexample} \label{ex-hill99}
Consider the MDP from \autoref{fig:hill99}, taken from \cite[Example 4.2]{Hill:99}.
Every FR-strategy attains only probability~$0$ of \Buchi.
Moreover, the strategy that, in state~$s_0$, subsequently picks $r_1, r_2, \ldots$ also attains probability~$0$, unlike in \autoref{ex-infmem-Buchi}.
But a different infinite-memory strategy achieves a positive probability.
Indeed, let $\zstrat$ be the strategy that, in~$s_0$, picks $2^1$ times $r_1$ and then $2^2$ times $r_2$ and \ldots $2^i$ times $r_i$ etc.
This strategy~$\zstrat$ achieves a positive probability of \Buchi.
(In more detail, $\zstrat$ achieves a positive probability of not falling in a losing sink~$\bot$, and in almost all of the remaining runs it visits a goal state infinitely often.)
Note that $\zstrat$ is a Markov strategy.
\end{ourexample}

\medskip
{\bf\noindent The open problem.}
While the MDPs in Examples \ref{ex-infmem-Buchi} and~\ref{ex-hill99} require infinite memory, Markov strategies suffice for them.
Such examples led to the question whether there always exists a family of
$\eps$-optimal Markov strategies for \Buchi\ in all countably infinite
MDPs.

A partial answer was given by Hill \cite{Hill:79} (Proposition 5.1), who showed that
$\eps$-optimal Markov strategies for \Buchi\ exist in the special case where the MDP
contains only a \emph{finite} number of controlled states.
This result applies to the MDPs from \autoref{ex-hill99}
and \autoref{fig:infmem-Buchi}b),
but not directly to the one in \autoref{fig:infmem-Buchi}a).

The question for general MDPs was stated as an open problem in \cite{Hill:79}
(p.158, l.4) and mentioned again in \cite{Hill:99} (Q1 in Section 5).

\medskip
{\bf\noindent Our contributions.}
We provide a negative answer to the open problem. We construct a non-trivial example
of a countable acyclic and finitely branching MDP and prove that no
$\eps$-optimal Markov strategies
for \Buchi\ exist for it (for any $\eps <1$).
In combination with the example from \cref{fig:infmem-Buchi}, this shows that for
general MDPs neither finite memory (FR) nor Markov strategies are sufficient.

Secondly, we provide an upper bound on the strategy complexity of \Buchi.
We show that for \emph{acyclic} countable MDPs there always exist $\eps$-optimal
strategies 
that are deterministic and use only one bit of memory.
Since every MDP can be transformed into an acyclic one by encoding a step-counter into the
states, it follows that general countable MDPs have $\eps$-optimal strategies for \Buchi\ that are deterministic and use only a step-counter plus one extra bit of memory.
Thus Markov strategies are almost, but not quite, sufficient.
\cref{tab:results} summarizes these results.

{
\newcommand{\ok}{Y}
\newcommand{\newok}{\textbf{Y}}
\newcommand{\no}{N}
\newcommand{\newno}{\textbf{N}}

\begin{table*}[ht]
\centering
\begin{tabular}{|l|c|c|c|c|c|}
\hline
\mbox{$\eps$-optimal strategy for \Buchi}     & MD  & 1-bit D & FR  & Markov & Markov+1 bit D  \\ \hline
\mbox{Finite MDP}                     & \ok & \ok     & \ok & \ok    & \ok             \\  \hline
\mbox{MDP w.\ finitely many controlled states} & \no & \no     & \no & \ok    & \ok             \\  \hline
\mbox{\vspace*{2mm}Acyclic MDP}                    & \newno & \newok     & \newok & \newno    & \newok             \\  \hline
\mbox{General MDP}                    & \no & \no     & \no & \newno    & \newok             \\  \hline
\end{tabular}
\smallskip
\caption{Existence of various types of $\eps$-optimal strategies for the \Buchi\ objective,
for several classes of MDPs. New results are in boldface.}\label{tab:results}
\end{table*}
}

\section{Preliminaries}\label{sec:prelim}

A \emph{probability distribution} over a countable set $S$ is a function
$f:\states\to[0,1]$ with $\sum_{\state\in \states}f(\state)=1$.
We write
$\dist(\states)$ for the set of all probability distributions over $\states$.

For a set $S$ we write $S^*$ (resp.\ $S^\omega$) for the set of all finite (resp.\ infinite) sequences of elements in $S$. 
We use slightly generalized regular expressions for sets of sequences, e.g., if $s_0 \in S$ we may write $s_0 S^\omega$ for the set of infinite sequences starting with~$s_0$.

\medskip
\noindent{\bf Markov decision processes.}
A \emph{Markov decision process} (MDP) $\mdp=\mdptuple$ consists of
a countable set~$\states$ of \emph{states},
which is partitioned into a set~$\zstates$ of \emph{controlled states}
and  a set~$\rstates$ of \emph{random states},
a  \emph{transition relation} $\transition\subseteq\states\x\states$,
and a  \emph{probability function}~$\probp:\rstates \to \dist(\states)$.
We  write $\state\transition{}\state'$ if $\tuple{\state,\state'}\in \transition$,
and  refer to~$s'$ as a \emph{successor} of~$s$.
We assume that every state has at least one successor.
The probability function~$P$  assigns to each random state~$\state\in \rstates$
a probability distribution~$P(\state)$ over its (non-empty) set of successor states.
A \emph{sink in $\mdp$} is a subset $T \subseteq \states$ closed under the $\transition$ relation,
that is,  $\state \in \reachset$ and  $\state\transition\state'$ implies that $\state'\in T$.

An MDP is \emph{acyclic} if the underlying directed graph~$(S,\transition)$ is acyclic, i.e.,
there is no directed cycle.
It is  \emph{finitely branching}
if every state has finitely many successors
and \emph{infinitely branching} otherwise.
An MDP without controlled states
($\zstates=\emptyset$) is called a \emph{Markov chain}.


\medskip
\noindent{\bf Strategies and Probability Measures.}
A \emph{run}~$\play$ is an  infinite sequence $\state_0\state_1\cdots$ of states
such that $\state_i\transition{}\state_{i+1}$ for all~$i\in \mathbb{N}$;
write~$\play(i)\eqdef\state_i$ for the $i$-th state along~$\play$.
A \emph{partial run} is a finite prefix of a run.
We say that (partial) run $\play$ \emph{visits} $\state$ if
$\state=\play(i)$ for some $i$, and that~$\play$ starts in~$s$ if $\state=\play(0)$.

A \emph{strategy} 
is a function $\zstrat:\states^*\zstates \to \dist(S)$ that
assigns to partial runs $\partialplay\state \in \states^*\zstates$
a distribution over the successors~$\{\state'\in \states\mid \state \transition{} \state'\}$.
The set of all strategies  in $\mdp$ is denoted by $\zstratset_\mdp$
(we omit the subscript and write~$\zstratset$ if $\mdp$ is clear from the context).
A (partial) run~$\state_0\state_1\cdots$ is induced by strategy~$\zstrat$
if for all~$i$
either $\state_i \in \zstates$ and $\zstrat(\state_0\state_1\cdots\state_i)(\state_{i+1})>0$,
or
$\state_i \in \rstates$ and $\probp(\state_i)(\state_{i+1})>0$.
%

An MDP $\mdp=\mdptuple$, an initial state $\state_0\in \states$, and a strategy~$\zstrat$
induce a probability space in which the outcomes are runs starting in $\state_0$
and with measure $\probm_{\mdp,\state_0,\zstrat}$
defined as follows.
It is first defined on \emph{cylinders} $s_0 s_1 \ldots s_n \states^\omega$, where $s_1, \ldots, s_n \in \states$:
if $s_0 s_1 \ldots s_n$ is not a partial run induced by~$\zstrat$ then
$\probm_{\mdp,\state_0,\zstrat}(s_0 s_1 \ldots s_n \states^\omega) \eqdef 0$.
Otherwise, $\probm_{\mdp,\state_0,\zstrat}(s_0 s_1 \ldots s_n \states^\omega) \eqdef \prod_{i=0}^{n-1} \bar{\zstrat}(s_0 s_1 \ldots s_i)(s_{i+1})$, where $\bar{\zstrat}$ is the map that extends~$\zstrat$ by $\bar{\zstrat}(w s) = \probp(s)$ for all $w s \in \states^* \rstates$.
By Carath\'eodory's theorem~\cite{billingsley-1995-probability},
this extends uniquely to a probability measure~$\probm_{\mdp,\state_0,\zstrat}$ on
the Borel $\sigma$-algebra $\?F$ of subsets of~$s_0 \states^\omega$.
Elements of $\?F$, i.e., measurable sets of runs, are called \emph{events} or \emph{objectives} here.
For $X\in\?F$ we will write $\complementof{X}\eqdef s_0S^\omega\setminus X\in \?F$ for its complement
and $\expectval_{\mdp,\state_0,\zstrat}$
for the expectation w.r.t.~$\probm_{\mdp,\state_0,\zstrat}$.
We drop the indices wherever possible without introducing ambiguity.

\medskip
\noindent{\bf  Strategy Classes.}
Strategies are in general  \emph{randomized} (R) in the sense that they take values in $\dist(\states)$.
A strategy~$\zstrat$ is \emph{deterministic} (D) if $\zstrat(\rho)$ is a Dirac distribution
for all runs~$\rho\in \states^{*} \zstates$.


We formalize the amount of \emph{memory} needed to implement strategies.
Let $\memory$ be a countable set of memory modes, and let $\tau: \memory\times \states \to \dist(\memory\times \states)$ be a function that meets the
following two conditions: for all modes $\memconf \in \memory$,
\begin{itemize}
	\item for all controlled states~$\state\in \zstates$,
	the distribution $\tau(\memconf,\state)$ is   over
	$\memory \times \{\state'\mid \state \transition{} \state'\}$.
	\item for all random states~$\state \in \rstates$, we have $\sum_{\memconf'\in \memory} \tau(\memconf,\state)(\memconf',\state')=P(\state)(\state')$.
\end{itemize}

The function~$\tau$ together with an initial memory mode~$\memconf_0$
induce a strategy~$\zstrat_{\tau}:\states^*\zstates \to \dist(S)$ as follows.
Consider the Markov chain with the set~$\memory \times \states$ of states
and the  probability function~$\tau$.
A sequence $\rho=s_0 \cdots s_i$ corresponds to a set
$H(\rho)=\{(\memconf_0,s_0) \cdots (\memconf_i,s_i) \mid \memconf_0,\ldots, \memconf_i\in \memory\}$
of runs in this Markov chain.
Each $\rho s \in \state_0 \states^{*} \zstates$
induces a
probability distribution~$\mu_{\rho \state}\in \dist(\memory)$,
the  probability of   being in  state~$(\memconf,s)$
conditioned on having  taken some partial run
from~$H(\rho s)$.
We define~$\zstrat_{\tau}$ such that
$\zstrat_{\tau}(\rho \state)(\state')=\sum_{\memconf,\memconf'\in \memory} \mu_{\rho \state}(\memconf) \tau(\memconf,\state)(\memconf',\state') $
for all $\rho \state\in \states^{*} \zstates$ and all $\state' \in \states$.

We say that a strategy $\zstrat$ can be \emph{implemented} with
memory~$\memory$ if there exist~$\memconf_0 \in \memory$
and $\tau$ such that  $\zstrat_{\tau}=\zstrat$.
We define certain classes of strategies:
\begin{itemize}
\item 
A strategy~$\zstrat$ is \emph{finite memory}~(F) if
there exists a finite memory~$\memory$ implementing~$\zstrat$.

\item A strategy $\zstrat$ is  \emph{memoryless}~(M) (also called \emph{positional})
if it can be implemented with a memory  of size~$1$.
We may view
M-strategies as functions $\zstrat: \zstates \to \dist(\states)$.

\item A strategy $\zstrat$ is  \emph{1-bit}
if it can be implemented with a memory  of size~$2$.
Such a strategy is then determined by a function
$\tau:\{0,1\}\times \states \to \dist(\{0,1\} \times \states)$.
Intuitively~$\tau$ uses one bit of memory 
to capture two different modes. 

\item A strategy~$\zstrat$ is \emph{Markov} if
 it can be implemented with the natural numbers $\mathbb{N}$
as the memory, and a function $\tau$ such that the distribution
$\tau(\memconf,\state)$ is over $\{\memconf+1\}\times \states$
for all $\memconf\in \memory$ and $\state\in \states$.
Intuitively, such a  strategy depends only on the
the current state and the number of steps taken so far,
i.e., it has access to a step-counter.
We  view Markov strategies as functions
$\zstrat: \mathbb{N} \times \zstates \to \dist(\states)$.
Note that such a strategy is generally not finite memory.

\item A strategy~$\zstrat$ is \emph{1-bit Markov} if
 it can be implemented with $\mathbb{N} \times \{0,1\}$
as the memory, and a function $\tau$ such that the distribution
$\tau(n,b,\state)$ is over $\{n+1\}\times \{0,1\}\times \states$
for all $(n,b)\in \memory$ and $\state\in \states$.
We  view such strategies as functions
$\zstrat: \mathbb{N}\times \{0,1\} \times \zstates \to \dist(\{0,1\} \times \states)$.
\end{itemize}


\newcommand{\payoff}{f}
\newcommand{\reward}{\mathit{r}}
\medskip
\noindent {\bf Payoffs, Values, Optimality.}
We are interested in strategies to maximize the expectation of a given
measurable \emph{payoff} function
$\payoff:\states^\omega\to \R$,
a random variable that assigns a real value to every run.
The \emph{value} of state $\state$ (w.r.t.~$f$) is
the supremum of expected values of $\payoff$ over all strategies:
$$
\valueof{\mdp,\payoff}{\state} \eqdef \sup_{\zstrat\in\zstratset}\expectval_{\mdp,\state,\zstrat}(\payoff),
$$
For $\eps \ge 0$ and $\state \in \states$, we say that
a strategy $\zstrat$ is \emph{$\eps$-optimal}  iff
$\expectval_{\mdp,\state,\zstrat}(\payoff) \ge \valueof{\mdp,\payoff}{\state}
-\eps$
and \emph{uniformly} $\eps$-optimal iff this holds for every $\state\in\states$.
A (uniformly) $0$-optimal strategy is simply called (uniformly) \emph{optimal}.

In this paper, we will need two types of payoff functions.
The first is the \emph{total reward}, a random variable given as $\payoff(\rho) \eqdef\sum_{t=0}^\infty \reward(\rho(t))$,
where $\reward:\states\to\R$ is some given \emph{reward} function.
A useful fact \cite[Theorem 7.1.9]{Puterman:book} is that if $\states$ is finite and the range of $\reward$ is bounded then there exist optimal strategies (for total reward) which are memoryless and deterministic.

The second type of payoff functions we consider are
those with range $\{0,1\}$. Each such payoff function~$\payoff$ uniquely identifies an objective (set of runs)~$\formula$ by viewing~$f$ as the characteristic function of~$\formula$, i.e., $\payoff(\rho)=1$ if $\rho\in \formula$ and $0$ otherwise.
Then 
$\expectval_{\mdp,\state,\zstrat}(\payoff) = \probm_{\mdp,\state,\zstrat}(\formula)$.
We call this the \emph{probability of achieving} $\formula$ (using strategy~$\zstrat$ starting from the state~$\state$)
and simply write $\valueof{\mdp,\formula}{\state} = \valueof{\mdp,\payoff}{\state} = \sup_{\zstrat\in\zstratset}\probm_{\mdp,\state,\zstrat}(\formula)$.

%
Our main focus are \emph{reachability} (sometimes also called \emph{goal}) and \emph{\Buchi} objectives, which are determined by a set of states $F\subseteq S$ and defined as follows.
Let us slightly abuse notation and identify $F$ with its characteristic function, i.e., $F(s)=1$ if $s\in F$.
\begin{itemize}
    \item
The \emph{reachability} objective is to visit $F$ at least once during a run. The corresponding payoff is $\payoff(\play)\eqdef \max_{t\in \N}\play(t)$,
and we define
$\setf(F) \eqdef\{\play \in \states^{\omega} \mid \max_{t\in\N}F(\play(t))=1\}$;

\item
The \emph{\Buchi} objective is to visit $F$ infinitely often. The corresponding payoff function is
$\payoff(\play)\eqdef \limsup_{t\to\infty}F(\rho(t))$, and we let
$\reset(F)\eqdef\{\play \in \states^{\omega}\mid \limsup_{t\to \infty} F(\play(t))=1\}.$
\end{itemize}

\section{The Lower Bound}\label{sec:lower}
In this section we solve Hill's problem (\cite{Hill:79}
and \cite[Q1]{Hill:99}) by exhibiting an MDP
where the initial state has value $1$ w.r.t.\ the \Buchi\ objective,
but every Markov strategy achieves this objective with probability~$0$.
As explained in the introduction, it follows that in acyclic MDPs, $\eps$-optimal MR-strategies are not guaranteed to exist.
In fact, in the following theorem we prove the latter fact first, and subsequently generalize it to solve Hill's problem.

\begin{restatable}{ourtheorem}{thmnoMarkov} \label{thm-no-Markov}
There exists an acyclic MDP~$\mdp$, a state~$\state_0$ and a set of states~$F$ such that
\begin{enumerate}
\item
for every Markov strategy~$\zstrat$, we have $\probm_{\mdp,\state_0,\zstrat}(\reset(F)) = 0$, and
\item
$\valueof{\reset(F)}{\state_0} = 1$ and for every $\eps>0$
there exists a deterministic 1-bit strategy~$\zstrat_\eps$ s.t.~
$\probm_{\mdp,\state_0,\zstrat_\eps}(\reset(F)) \ge 1-\eps$.
\end{enumerate}
\end{restatable}

\noindent
In the remainder of this section we provide a proof sketch.
The full proof is in \cref{app:lower}.

\begin{proof}[Proof sketch for \autoref{thm-no-Markov}]
Our construction is based on an infinite MDP~$\M$ that consists of a chain of height-$n$ trees, $T^n$, for $n \in \N = \{1, 2, \ldots\}$.
\autoref{fig-no-MR} depicts its initial segment $T^1, T^2, T^3$.
Each such tree is ``rooted'' at a brown state on the top level, with a transition incoming from a blue state.
We make use of some conventions that simplify the presentation and the analysis.
In \autoref{fig-no-MR}, the different colors of the states highlight the structure of the MDP; the colors are also indicated by letters in the states: blue (L), brown (B), yellow (Y), red (R), green (G), white (W).
The start state, $\state_0$, is the blue state in the top-left corner.
The controlled states are exactly the yellow states.
The goal set $F$ consists of the green states at the bottom.
Two transitions emanate from each red state: a black (right) transition and a red (left) transition, both leading to the same (brown or green) state.

We consider the strengthened \Buchi\ objective that asks to see $F$ infinitely often and moreover that
\emph{no red transition} is taken. This corresponds exactly to the normal \Buchi\ objective if we redirect every red transition to an infinite (losing) chain of non-green states (not depicted in \autoref{fig-no-MR}).

We first argue that no MR-strategy achieves a positive probability of that objective.
Then we show that the MDP~$\M$ can be modified so that no Markov strategy achieves a positive probability.

\begin{sidewaysfigure}
\scalebox{0.85}{
\mbox{}\hspace{-15mm}
\begin{tikzpicture}[auto,inner sep=1,xscale=1.3,yscale=1.5]
\onebit
\end{tikzpicture}
}
\caption{For this acyclic MDP~$\M$ there are $\eps$-optimal deterministic
    1-bit strategies for the \Buchi\ objective $\reset(\reachset)$ where $\reachset$ contains exactly all green states.
    No MR-strategy achieves even a positive probability.
}
\label{fig-no-MR}
\end{sidewaysfigure}
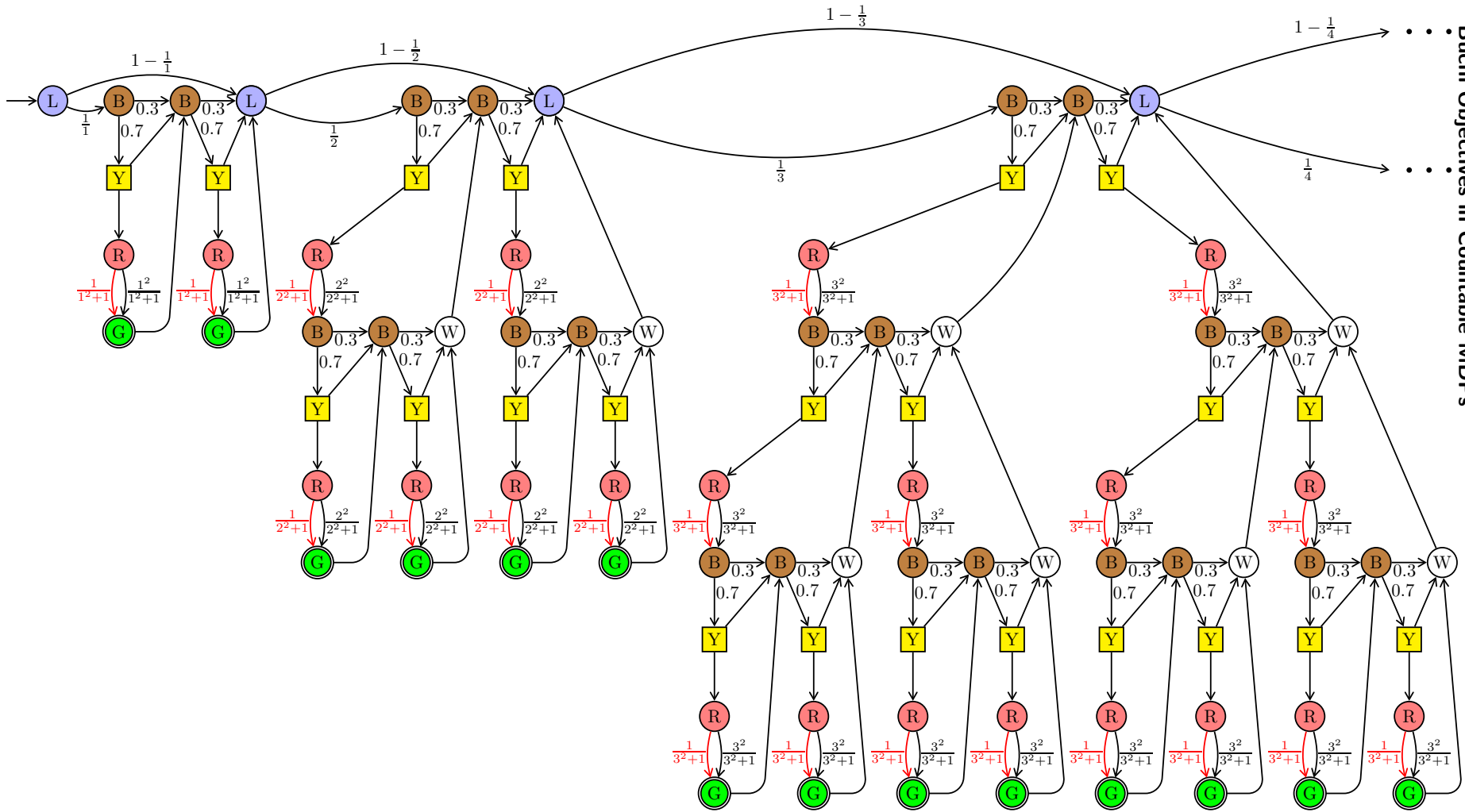

\medskip
\noindent
\smallparg{Intuition behind the construction of~$\M$.}
The objective, say $\formula$, of visiting infinitely many green states and no red transition creates tension between trying to visit green states and avoiding too many red states (the latter states incur a risk of taking a red transition).
In the proof we need to show that no memoryless strategy strikes a good balance between these competing goals.
On the one end of the spectrum, an MR-strategy might always choose the upward transition in the yellow states (which are the only controlled states).
But such a strategy never visits any green state, thus clearly violates~$\formula$.
On the other end of the spectrum lies the ``greedy'' MR-strategy, which always chooses the downward transition in the yellow states, in order to visit as many green states as possible.
Indeed, under this strategy, let $u_n$ denote the probability that, starting in the top-left brown state of~$T^n$, no green state is visited in~$T^n$.
By induction (given in the appendix as part of the general proof) one can show that there is $u<1$ such that $u_n \le u$ holds for all~$n$.
Considering the probability of the transitions emanating from the blue states (at the top), the expected overall number of visited green states is at least $\sum_{n = 1}^\infty \frac1{n} (1-u_n) \ge (1-u) \sum_{n = 1}^\infty \frac1{n} = \infty$.
It is not hard to strengthen this statement so that the greedy strategy almost surely visits infinitely many green states.
So the greedy strategy satisfies one part of~$\formula$, but it does so at the expense of visiting many red states.
Red states though are associated with a risk of taking a red transition, and it follows from the proof in the appendix that the greedy strategy almost surely ends up taking at least one (and indeed infinitely many) red transition(s).

\medskip
\noindent
\smallparg{Good 1-bit strategies.}
The two competing goals discussed in the previous paragraph can be balanced using a deterministic 1-bit strategy, which we describe in the following.
This strategy, $\zstrat_1$, sets its bit to~$0$ whenever a blue state (at the top) is entered.
While the bit is~$0$, in each tree~$T^n$ it maximizes the probability of visiting a green state by choosing the downward transition in the yellow states, thus accepting a certain risk of taking a red transition.
However, if and when a green state in~$T^n$ is visited, the bit is set to~$1$, and for the remaining sojourn in~$T^n$ the strategy~$\zstrat_1$ chooses the upward transitions in the yellow states, thus avoiding any risk of a red transition in the remainder of~$T^n$.
Although $\zstrat_1$ appears to visit fewer green states than the aforementioned ``greedy'' MR-strategy, $\zstrat_1$ still visits infinitely many green states almost surely.
This is because for each tree~$T^n$, the two strategies have the same probability of visiting at least one green state in~$T^n$.
The strategy $\zstrat_1$ can be improved, for each $\eps>0$, to achieve~$\formula$ with probability at least $1-\eps$, by fixing the bit to~$1$ in the first $k$ trees $T^1, \ldots, T^k$, for a $k$ that depends on~$\eps$.
Thus the first $k$ trees are virtually skipped, eliminating the risk of taking any red transition there.
In this way one can make the risk of taking a red transition arbitrarily small, while still visiting infinitely many green states with probability~$1$.

\medskip
\noindent
\smallparg{No good MR-strategies.}
We need to show that not only the extreme MR-strategies described above are inadequate but that every MR-strategy achieves~$\formula$ with probability~$0$.
To this end, for each tree $T^n$, define two probabilities:
\begin{itemize}
\item $t_n$ (for ``total success''): the probability that, starting in the top-left brown state of $T^n$, at least one green state but no red transition is visited in~$T^n$;
\item $d_n$ (for ``death''): the probability that, starting in the top-left brown state of $T^n$, a red transition is visited in~$T^n$.
\end{itemize}
A very technical proof shows that $d_n \ge 0.008 \cdot t_n$ holds for all~$n$, and this key inequality captures the inability of \emph{any} MR-strategy to strike an adequate balance between the mentioned competing goals.
Indeed, one can show that for an MR-strategy to have a positive probability of not visiting any red transition, the series $\sum_{n=1}^\infty \frac1{n} \cdot d_n$ needs to converge;
but to have a positive probability of visiting infinitely many green states, the series $\sum_{n=1}^\infty \frac1{n} \cdot t_n$ needs to diverge (in both cases, the factor $\frac1n$ is the probability of visiting the top-left brown node of~$T^n$).
By the inequality above, this is impossible.

\medskip
\noindent
\smallparg{No good Markov strategies.}
For the proof of \autoref{thm-no-Markov}, we also need to show that all Markov strategies achieve probability~$0$.
To this end, we modify the MDP~$\M$ so that for each state, all paths from the initial state~$s_0$ to~$s$ have the same length.
This can be achieved by replacing some transitions in~$\M$ by longer chains consisting of non-green states.
This modification does not change the fact that MR-strategies achieve probability~$0$.
But since in the new MDP each state can only be visited at a certain time, which is known a priori, a step-counter does not help.
Hence all Markov strategies, like MR-strategies, achieve~$\formula$ with probability~$0$.
\end{proof}
\autoref{thm-no-Markov} answers Hill's question negatively.
By combining the MDP from \autoref{thm-no-Markov} with one of the MDPs
from \autoref{fig:infmem-Buchi} (by adding a new initial random state that
branches to the MDPs with probability $\frac12$ each), one can even construct a single
MDP whose value w.r.t.~$\reset(F)$ is~$1$,
but every FR- and every Markov strategy achieves probability~$0$.

A slight modification of the example above yields a lower bound on the memory
requirements for the almost-sure parity objective.
Recall that the parity objective is defined on systems whose states are
labeled by a finite set of colors $C \eqdef \{1,2,\dots,{\it max}\} \subseteq \N$,
where a run is in $\Parity{C}$ iff the highest color that is seen infinitely
often in the run is even.

\begin{restatable}{ourcorollary}{corParity}\label{cor-123Parity}
There exist an acyclic MDP~$\mdp'$ with colors $\{1,2,3\}$ and a state~$\state_0$ such that
\begin{enumerate}
\item
for every Markov strategy~$\zstrat$, we have $\probm_{\mdp',\state_0,\zstrat}(\Parity{\{1,2,3\}}) = 0$, and
\item
there exists a deterministic 1-bit strategy $\zstrat'$ such that
$\probm_{\mdp',\state_0,\zstrat'}(\Parity{\{1,2,3\}}) = 1$.
\end{enumerate}
\end{restatable}
\begin{proof}
We obtain $\mdp'$ by modifying the MDP $\mdp$ from
\autoref{thm-no-Markov} as follows.
Label all green states in $F$ by color $2$ and the rest by color $1$.
Then modify each red transition to go to its target via a fresh state labeled by color $3$.
Clearly $\mdp'$ is still acyclic and labeled by colors $\{1,2,3\}$.

From the proof of \autoref{thm-no-Markov} (1), under every Markov strategy in
$\mdp$ a.s.\ seeing infinitely many green states (in $F$) implies seeing infinitely
many red transitions.
So in $\mdp'$ every Markov strategy $\zstrat$ a.s.\ either sees color $2$ only finitely often or
color $3$ infinitely often, thus $\probm_{\mdp',\state_0,\zstrat}(\Parity{\{1,2,3\}}) = 0$.

From the proof of \autoref{thm-no-Markov} (2), there is a deterministic 1-bit strategy
$\zstrat$ in $\mdp$ that attains probability $\ge 1/2$ for $\reset(F)$
without taking any red transition
and otherwise a.s.\ takes a red transition.
This property of $\zstrat$ holds not only when starting from $\state_0$ but from every other
state as well. We obtain $\zstrat'$ in $\mdp'$ by
continuing to play $\zstrat$ even after red transitions have been taken.
Under $\zstrat'$ the probability of going through infinitely many red transitions
(and seeing color 3) is $\le (1/2)^\infty =0$, and the probability of seeing
infinitely many states in $F$ (with color $2$) is $1$.
Thus $\probm_{\mdp',\state_0,\zstrat'}(\Parity{\{1,2,3\}}) = 1$.
\end{proof}

\section{The Upper Bound}\label{sec:upper}
We show that acyclic MDPs admit $\eps$-optimal deterministic 1-bit strategies for \Buchi.

We start by giving some intuition why 1 bit of memory is needed
and how it is used.
A step $\state' \transition \state''$ from some controlled state $\state'$
is \emph{value-decreasing} iff $\valueof{}{\state''} < \valueof{}{\state'}$.
While an optimal strategy can never tolerate any value-decreasing step,
an $\eps$-optimal strategy might have to take value-decreasing steps
infinitely often. The trick is to keep the collective value-loss sufficiently
small ($\le \eps$), while satisfying the other requirements of the objective.
So the strategy needs to play `ever better' (i.e., tolerate
only smaller and smaller value decreases) along a run.
In general this requires infinite memory, since one might re-visit the same
state infinitely often and needs to
choose a different transition from it every time; cf.\ \cref{fig:infmem-Buchi}.
However, in an acyclic MDP, with high probability, the distance to the initial state increases with the number of steps taken.
Thus one can partition the state space into separate regions, depending on
the distance from the initial state, and fix an acceptable rate of
value-decrease for each region.
Just limiting the collective value-loss is not sufficient for
\Buchi, one also needs to make progress and visit the set of goal states $F$ at least
once in each region.
The problem is that some runs might linger in some region too long, and visit
$F$ \emph{many times}, but
see too many value-decreasing steps at the rate of this region.
Therefore, as soon as one has visited $F$ in some region, one should try to get
to the next outer region (further away from the initial state) where the rate of
value-loss is smaller. Thus one needs 1 bit of memory to record whether one
has already seen $F$ in this region.
(Remember that the same state can be reached by different runs with different
histories.)
Just 1 bit suffices, because the probability of returning to a previous inner
region (and misinterpreting the bit) can be made arbitrarily small, since the
MDP is acyclic.

\begin{restatable}{ourtheorem}{thmMDPonebitBuchi}\label{thm:MDP-one-bit-Buchi}
For every acyclic countable MDP $\mdp$, finite set of initial states $\initstates$, set of
states $F$ and $\eps>0$, there exists a deterministic 1-bit
strategy for $\reset(F)$ that is $\eps$-optimal from every $\state \in \initstates$.
\end{restatable}
\begin{proof}
  Let $\mdp=\mdptuple$ be an acyclic MDP,
  $\initstates \subseteq \states$ a finite set of initial states and $F \subseteq \states$ a 
  set of goal states and $\formula \eqdef \reset(F)$
  denote the \Buchi\ objective w.r.t.\ $F$.
  We prove the claim for finitely branching $\mdp$ first and transfer the result to general MDPs at the end.
  \renewcommand{\optval}[1][\formula]{\valueof{\mdp,#1}{\state}}
For every $\eps >0$ and every $\state \in \initstates$ there exists an $\eps$-optimal strategy $\zstrat_\state$
such that
\begin{equation}\label{inbody:eq:eps-opt}
\probm_{\mdp,\state,\zstrat_\state}(\formula) \ge \optval-\eps.
\end{equation}
However, the strategies $\zstrat_\state$ might differ from each other and
might use randomization and a large (or even infinite)
amount of memory.
We will construct a single deterministic strategy $\zstrat'$ that uses only 1 bit of
memory such that $\forall_{\state \in \initstates}\, \probm_{\mdp,\state,\zstrat'}(\formula) \ge \optval-2\eps$.
This proves the claim as $\eps$ can be chosen arbitrarily small.

In order to construct $\zstrat'$, we first observe the behavior of the
finitely many $\zstrat_\state$ for $\state \in \initstates$
on an infinite, increasing sequence of finite subsets of $\states$.
Based on this, we define a second stronger objective $\formula'$ with
\begin{equation}\label{inbody:eq:prime-implies-normal}
\formula' \subseteq \formula,
\end{equation}
and show that all $\zstrat_\state$ 
attain at least $\optval-2\eps$
w.r.t.\ $\formula'$, i.e.,
\begin{equation}\label{inbody:eq:observe-orig}
    \forall_{\state \in \initstates}\, \probm_{\mdp,\state,\zstrat_\state}(\formula') \ge \optval-2\eps.
\end{equation}
We construct $\zstrat'$ as a deterministic 1-bit \emph{optimal} strategy w.r.t.\ $\formula'$
from all $\state \in \initstates$ and obtain
\begin{align*}
\probm_{\mdp,\state,\zstrat'}(\formula)\ &
\ge\ \probm_{\mdp,\state,\zstrat'}(\formula') && \text{by \cref{inbody:eq:prime-implies-normal}} \\
  & \ge\ \probm_{\mdp,\state,\zstrat_\state}(\formula') && \text{by optimality of $\zstrat'$ for $\formula'$}\\
  & \ge \ \optval-2\eps && \text{by \cref{inbody:eq:observe-orig}}.
\end{align*}

\medskip
\noindent
\smallparg{Informal outline: Behavior of $\zstrat_\state$, objective $\formula'$ and properties
  \cref{inbody:eq:prime-implies-normal} and \cref{inbody:eq:observe-orig}.}
For the formal proof see \cref{app:upper}.

Let $\bubble{k}{\initstates}$ be the set of states that are reachable from some
initial state in $\initstates$ within at most $k$ steps.
Since $I$ is finite and $\mdp$ is finitely branching, $\bubble{k}{\initstates}$ is finite for
every~$k$.

We define a sequence of sufficiently large and increasing numbers $k_i$ and
$l_i$ with $k_i < l_i < k_{i+1}$ for $i \in \N$
and finite sets $K_i \eqdef \bubble{k_i}{\initstates}$
and $L_i \eqdef \bubble{l_i}{\initstates}$.
Every run from a $\state \in \initstates$ according to $\zstrat_\state$ must eventually leave each of
these finite sets, because $\mdp$ is acyclic.
Moreover, we choose these numbers so that once a run has left $L_i$ it is
very unlikely to return to $K_i$.
Let $F_{i} \eqdef F \cap K_{i} \setminus L_{i-1}$.
Runs according to $\zstrat_\state$ are very likely to follow a particular pattern.
Let $R_1 \eqdef (K_1 \setminus F_1)^*F_1$,
$R_2 \eqdef (K_2 \setminus F_2)^*F_2$
and
$R_{i+1} \eqdef (K_{i+1} \setminus (F_{i+1} \cup K_{i-1}))^*F_{i+1}$ for
$i \ge 2$.
We show that
\begin{equation}\label{inbody:eq:eps-bound}
    \forall_{\state \in \initstates}\,\probm_{\mdp,\state,\zstrat_\state}(\formula \cap \complementof{R_1 R_2 \ldots R_{i+1} (S \setminus K_i)^\omega})
\quad\le\quad \eps
\end{equation}
We now define the Borel objectives
$R_{\le i} \eqdef R_1 R_2 \dots R_i \states^\omega$ and
$\formula' \eqdef \bigcap_{i \in \N} R_{\le i}$.
Since $F_i \cap F_k = \emptyset$ for $i \neq k$ and $\formula'$ implies a visit
to the set $F_i$ for all $i \in \N$, we have
$\formula' \subseteq \formula$ and obtain \cref{inbody:eq:prime-implies-normal}.
Using \cref{inbody:eq:eps-bound}, we show that
$
\forall_{\state \in \initstates}\,\probm_{\mdp,\state,\zstrat_\state}(\formula') \ge \ \optval-2\eps
$ and thus obtain \cref{inbody:eq:observe-orig}.

\medskip
\noindent
\smallparg{Definition of the 1-bit strategy $\zstrat'$.}
We now define a deterministic 1-bit strategy $\zstrat'$ that is optimal for
objective $\formula'$ from every $\state \in \initstates$.
First we define certain ``suffix'' objectives of $\formula'$.
Recall that $R_i = (K_{i} \setminus (F_{i} \cup K_{i-2}))^*F_{i}$.
Let $R_{i,j} \eqdef R_i R_{i+1} \dots R_j \states^\omega$
and $R_{\ge i} \eqdef \bigcap_{j \ge i} R_{i,j}$.
Consider the
objectives $R_{\ge i+1}$ for runs that start in states $\state' \in F_i$.
For every state $\state' \in F_i$ we consider its value w.r.t.\ the
objective $R_{\ge i+1}$, i.e.,
$\valueof{\mdp,R_{\ge i+1}}{\state'} \eqdef \sup_{\hat{\zstrat}} \probm_{\mdp,\state',\hat{\zstrat}}(R_{\ge i+1})$.
For every $i \ge 1$ we consider the finite subspace $K_i\setminus K_{i-2}$.
In particular, it contains the sets $F_{i-1}$ and $F_i$.
We define a bounded total reward
objective $B_i$ for runs starting in $F_{i-1}$ as follows.
Runs that exit the subspace (either by leaving $K_i$ or by visiting $K_{i-2}$)
before visiting $F_i$ get reward $0$. All other runs must visit $F_i$
eventually (since $\mdp$ is acyclic and the subspace is finite).
When some run reaches the set $F_i$ \emph{for the first time}
in some state $\state'$ then this run gets the reward of $\valueof{\mdp,R_{\ge i+1}}{\state'}$.
Using \cite[Theorem 7.1.9]{Puterman:book}, we show that there exists a uniform optimal MD-strategy $\zstrat_i$ for $B_i$ on $K_i\setminus K_{i-2}$ in $\mdp$.

We now define $\zstrat'$ by combining
different MD-strategies $\zstrat_i$, depending on the current state and on the
value of the 1-bit memory.
The intuition is that the strategy $\zstrat'$ has two modes: normal-mode
and next-mode.
In a state $\state' \in K_i \setminus K_{i-1}$, if the memory is $i\pmod 2$
then the strategy is in normal-mode and plays towards reaching $F_i$.
Otherwise, the strategy is in next-mode and plays towards reaching $F_{i+1}$.

Initially $\zstrat'$ starts in a state $\state \in \initstates$ with the 1-bit memory set to
$1$.
We define the behavior of $\zstrat'$ in a state $\state' \in K_i \setminus K_{i-1}$
for every $i \ge 1$.
If the 1-bit memory is $i \pmod 2$ and $\state' \notin F_i$ then $\zstrat'$ plays like
$\zstrat_i$.
(Intuitively, one plays towards $F_i$, since one has not yet visited it.)
If the 1-bit memory is $i \pmod 2$ and $\state' \in F_i$ then the 1-bit memory
is set to $(i+1) \pmod 2$, and $\zstrat'$ plays like $\zstrat_{i+1}$.
(Intuitively, one records the fact that one has already seen $F_i$ and then
targets the next set $F_{i+1}$.)
If the 1-bit memory is $(i+1) \pmod 2$ then $\zstrat'$ plays like
$\zstrat_{i+1}$.
(Intuitively, one plays towards $F_{i+1}$, since one has already visited $F_i$.)


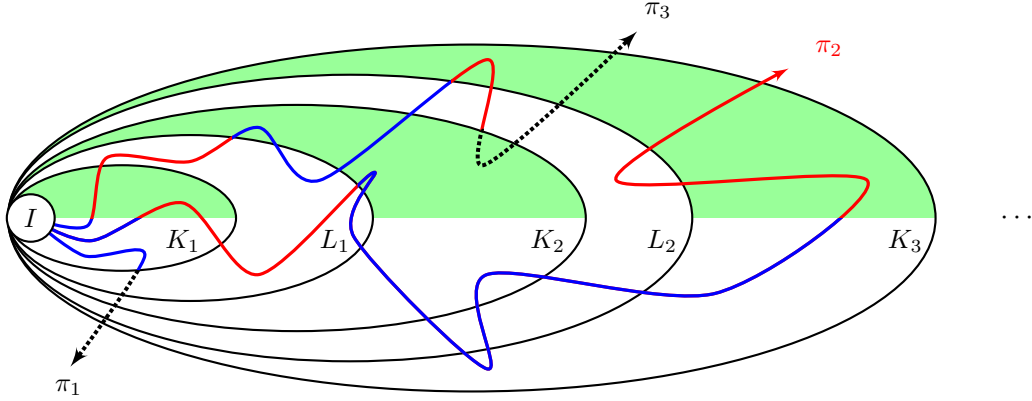
\begin{figure}[th]
   \begin{center}		
\centering

\makeatletter
\tikzset{
  use path for main/.code={%
    \tikz@addmode{%
      \expandafter\pgfsyssoftpath@setcurrentpath\csname tikz@intersect@path@name@#1\endcsname
    }%
  },
  use path for actions/.code={%
    \expandafter\def\expandafter\tikz@preactions\expandafter{\tikz@preactions\expandafter\let\expandafter\tikz@actions@path\csname tikz@intersect@path@name@#1\endcsname}%
  },
  use path/.style={%
    use path for main=#1,
    use path for actions=#1,
  }
}
\makeatother

\begin{tikzpicture}[>=latex',shorten >=1pt,node distance=1.9cm,on grid,auto,
fshade/.style={draw=none,fill=green!40!white,rotate=0},
bitone/.style={red,very thick},
bitzero/.style={blue, very thick},
lost/.style={densely dotted, ->, line width=1.5pt},
roundnode/.style={circle, draw,minimum size=1.5mm},
squarenode/.style={rectangle, draw,minimum size=2mm},
cross/.style={cross out,  fill=none, minimum size=2*(#1-\pgflinewidth), inner sep=0pt, outer sep=0pt}, cross/.default={1pt}]

\begin{scope}
\clip (5.8,0) ellipse (6.1cm and 2.3cm);
\draw[fshade] (-1,12) rectangle (12,0);
\end{scope}
\draw (5.8,0) ellipse (6.1cm and 2.3cm);

\draw [fill=white](4.2,0) ellipse (4.5cm and 1.9cm);

\begin{scope}
\clip (3.5,0) ellipse (3.8cm and 1.5cm);
\draw[fshade] (-1,10) rectangle (10,0);
\end{scope}
\draw (3.5,0) ellipse (3.8cm and 1.5cm);

\draw [fill=white](2.1,0) ellipse (2.4cm and 1.1cm);

\begin{scope}
\clip (1.2,0) ellipse (1.5cm and .7cm);
\draw[fshade] (-1,2) rectangle (3,0);
\end{scope}
\draw (1.2,0) ellipse (1.5cm and .7cm);


\node[draw=none](dot1) at (13,0)  {{\large $\cdots$}};
\node[draw=none](K1) at (2,-.3)   {$K_1$};
\node[draw=none](L1) at (4,-.3)   {$L_1$};
\node[draw=none](K2) at (6.8,-.3) {$K_2$};
\node[draw=none](L2) at (8.3,-.3) {$L_2$};
\node[draw=none](K3) at (11.5,-.3){$K_3$};

\coordinate (s) at (0,0);

\path[name path=firstrun] plot [smooth] coordinates {
    (s)
    (0.75,-.1)
    (1,0.8)
    (2.1,.75)
    (3,1.2)
    (3.8,.5)
    (6,2.1)
    (5.95,0.7)
    (8,2.5)
};
\begin{scope}
    \clip (s) rectangle (1,-0.5);
    \draw[bitzero,use path=firstrun];
\end{scope}
\begin{scope}
    \clip(s) rectangle (2.65,1.2);
    \draw[bitone,use path=firstrun];
\end{scope}
\begin{scope}
    \clip(2.65,-1.2) rectangle (5.54,2);
    \draw[bitzero,use path=firstrun];
\end{scope}
\begin{scope}
    \clip(5.54,2.5) rectangle (6.5,1.15);
    \draw[bitone,use path=firstrun];
\end{scope}
\begin{scope}
    \clip (5.9,1.15) -- (5,0) -- (7,0) -- (9,3) -- (8,3) --cycle;
    \draw[->,lost,use path=firstrun];
\end{scope}
\node[lost] at (8.25,2.75){$\pi_3$};

\draw[->,name path=secondrun,bitone] plot [smooth] coordinates {
    (s)
    (0.8,-.3)
    (2,.2)
    (3,-0.75)
    (4.5,0.6)
    (4.25,-0.25)
    (6,-2)
    (6,-0.75)
    (9,-1)
    (11,0.5)
    (7.7,.5)
    (10,2)
};
\begin{scope}
    \clip (0,0) -- (2,0) -- (1,-1.5) --cycle;
    \draw[bitzero,use path=secondrun];
\end{scope}
\begin{scope}
\clip(4.3,0.4) -- (4.25,1.5) -- (8,0) -- (10.9,0) -- (8,-2)-- (5,-2.1) -- (4,-0.5) --cycle;
    \draw[bitzero,use path=secondrun];
\end{scope}
\node[bitone] at (10.5,2.25){$\pi_2$};

\path [name path=thirdrun] plot [smooth] coordinates {
    (s)
    (0.7,-.5)
    (1.5,-0.5)
    (0.5,-2)
};
\begin{scope}
\clip (1.2,0) ellipse (1.5cm and .7cm);
    \draw[bitzero,use path=thirdrun];
\end{scope}
\begin{scope}
    \clip (0,-0.7) rectangle (2,-2);
    \draw[lost,->,use path=thirdrun];
\end{scope}
\node[lost] at (0.5,-2.25){$\pi_1$};

\node [roundnode,fill=white] at (s) {$I$};
\end{tikzpicture}
		\end{center}
                \caption{Memory updates along runs $\pi_1,\pi_2,\pi_3$,
                    drawn in blue while the memory-bit is one and in red while the bit is zero.
	            The green region in $K_1$ is $F_1$, and for all $i\geq 2$, the green region in $K_{i}\setminus L_{i-1}$ is $F_i$.
                    Both $\pi_1$ and $\pi_3$ violate $\varphi'$ and are drawn as dotted lines once they do.
}
\label{inbody:fig:flipingBit}
\end{figure}

Observe that if a run according to $\zstrat'$ exits some set $K_i$
(and thus enters $K_{i+1}\setminus K_i$) with the bit still set to $i \pmod 2$
(normal-mode) then this run has not visited $F_i$ and thus does not satisfy the objective
$\formula'$. (Or the same has happened earlier for some $j < i$, in which case
also the objective $\formula'$ is violated.)
An example is the run $\pi_1$ in \cref{inbody:fig:flipingBit}.
However, if a run according to $\zstrat'$ exits some set $K_i$
(and thus enters $K_{i+1} \setminus K_i$) with the bit set to $(i+1) \pmod 2$
(thus $\zstrat_{i+1}$ in next-mode)
then in the new set $K_{i'} \setminus K_{i'-1}$ with $i'=i+1$ the bit is set to
$i' \pmod 2$ and $\zstrat'$ continues to play like $\zstrat_{i+1}$ in normal-mode.
Even if this run returns (temporarily) to $K_i$ (but not to $K_{i-1}$)
the strategy $\zstrat'$ continues to play like $\zstrat_{i+1}$ in next-mode.
An example is the run $\pi_2$ in \cref{inbody:fig:flipingBit}.
Finally, if a run returns to $K_{i-1}$ after having visited $F_i$
then it fails the objective $\formula'$, e.g.,
run $\pi_3$ in \cref{inbody:fig:flipingBit}.

\medskip
\noindent
\smallparg{The 1-bit strategy $\zstrat'$ is optimal for $\formula'$ from every
$\state \in \initstates$.}
Let $\state \in \initstates$ be arbitrary.
For a given run from $\state$, let $\firstinset{F_i}$ be the first state
$\state'$ in $F_i$ that is visited (if any).
We define a bounded reward objective $B_i'$ for runs starting
at $\state$ as follows. Every run that does not satisfy the objective
$R_{\le i}$ gets assigned reward $0$.
Otherwise, consider a run from $\state$ that satisfies $R_{\le i}$.
When this run reaches the set $F_i$ for the first time
in some state $\state'$ then this run gets a reward of
$\valueof{\mdp,R_{\ge i+1}}{\state'}$. Note that this reward is $\le 1$.

We show that for all $i \in \N$
\begin{equation}\label{inbody:eq:Bi-prime}
\valueof{\mdp,\formula'}{\state} = \valueof{\mdp,B_i'}{\state}
\end{equation}
Towards the $\ge$ inequality,
let $\hat{\zstrat}$ be an $\hat{\eps}$-optimal strategy for
$B_i'$ from $\state$.
We define the strategy $\hat{\zstrat}'$ to play like $\hat{\zstrat}$
until a state $\state' \in F_i$ is reached and then to switch to
some $\hat{\eps}$-optimal strategy for objective $R_{\ge i+1}$
from $\state'$.
Every run from $\state$ that satisfies $\formula'$ can be split into parts,
before and after the first visit to the set $F_i$, i.e.,
$\formula' = \{w_1\state'w_2\ |\ w_1\state' \in R_{\le i}, \state' \in F_i,
\state'w_2 \in R_{\ge i+1}\}$.
Therefore we obtain that
$\probm_{\mdp,\state,\hat{\zstrat}'}(\formula') \ge
\expectval_{\mdp,\state,\hat{\zstrat}}(B_i') - \hat{\eps} \ge
\valueof{\mdp,B_i'}{\state} - 2\hat{\eps}$.
Since this holds for every
$\hat{\eps} >0$, we obtain
$\valueof{\mdp,\formula'}{\state} \ge \valueof{\mdp,B_i'}{\state}$.

Towards the $\le$ inequality,
let $\hat{\zstrat}$ be any strategy for
$\formula'$ from $\state$. We have
$\probm_{\mdp,\state,\hat{\zstrat}}(\formula')
\le
\sum_{\state' \in F_i}
\probm_{\mdp,\state,\hat{\zstrat}}(R_{\le i} \cap \firstinset{F_i}=\state')
\cdot \valueof{\mdp,R_{\ge i+1}}{\state'}
=
\expectval_{\mdp,\state,\hat{\zstrat}}(B_i')
$.
Thus $\valueof{\mdp,\formula'}{\state} \le \valueof{\mdp,B_i'}{\state}$.
Together we obtain \cref{inbody:eq:Bi-prime}.

For all $i \in \N$ and every state $\state' \in F_i$ we show that
\begin{equation}\label{inbody:eq:Ri-eq-Bi}
\valueof{\mdp,R_{\ge i+1}}{\state'} = \valueof{\mdp,B_{i+1}}{\state'}
\end{equation}
Towards the $\ge$ inequality,
let $\hat{\zstrat}$ be an $\hat{\eps}$-optimal strategy for
$B_{i+1}$ from $\state' \in F_i$.
We define the strategy $\hat{\zstrat}'$ to play like $\hat{\zstrat}$
until a state $\state'' \in F_{i+1}$ is reached and then to switch to
some $\hat{\eps}$-optimal strategy for objective $R_{\ge i+2}$
from $\state''$. We have that
$\probm_{\mdp,\state',\hat{\zstrat}'}(R_{\ge i+1}) \ge
\expectval_{\mdp,\state',\hat{\zstrat}}(B_{i+1}) - \hat{\eps} \ge
\valueof{\mdp,B_{i+1}}{\state} - 2\hat{\eps}$.
Since this holds for every $\hat{\eps} >0$, we obtain
$\valueof{\mdp,R_{\ge i+1}}{\state'} \ge \valueof{\mdp,B_{i+1}}{\state'}$.

Towards the $\le$ inequality,
let $\hat{\zstrat}$ be any strategy for
$R_{\ge i+1}$ from $\state' \in F_i$.
We have
\begin{align*}
\probm_{\mdp,\state',\hat{\zstrat}}(R_{\ge i+1})
~&\le
\sum_{\state'' \in F_{i+1}}
\probm_{\mdp,\state',\hat{\zstrat}}(R_{i+1}\states^\omega \cap \firstinset{F_{i+1}}=\state'')
\cdot \valueof{\mdp,R_{\ge i+2}}{\state''}\\
 &=
\expectval_{\mdp,\state',\hat{\zstrat}}(B_{i+1}).
\end{align*}
Thus $\valueof{\mdp,R_{\ge i+1}}{\state'} \le \valueof{\mdp,B_{i+1}}{\state'}$.
Together we obtain \cref{inbody:eq:Ri-eq-Bi}.

We show, by induction on $i$, that $\zstrat'$ is optimal for $B_i'$ for
all $i \in \N$ from start state $\state$, i.e.,
\begin{equation}\label{inbody:eq:opt-Bi-prime}
\expectval_{\mdp,\state,\zstrat'}(B_i') = \valueof{\mdp,B_i'}{\state}
\end{equation}
In the base case of $i=1$ we have that $B_1' = B_1$. The strategy $\zstrat'$ plays
$\zstrat_1$ until reaching $F_1$, which is optimal for objective $B_1$ and
thus optimal for $B_1'$.
For the induction step we assume (IH) that $\zstrat'$ is optimal for $B_i'$.
\begin{align*}
\valueof{\mdp,B_{i+1}'}{\state}\ & =\ \valueof{\mdp,B_i'}{\state} && \text{by \cref{inbody:eq:Bi-prime}}\\
                           & =\ \expectval_{\mdp,\state,\zstrat'}(B_i') && \text{by (IH)}\\
& =\ \sum_{\state' \in F_i}
\probm_{\mdp,\state,\zstrat'}(R_{\le i} \cap \firstinset{F_i}=\state')
\cdot \valueof{\mdp,R_{\ge i+1}}{\state'} && \text{by def.\ of $B_i'$}\\
& =\ \sum_{\state' \in F_i}
\probm_{\mdp,\state,\zstrat'}(R_{\le i} \cap \firstinset{F_i}=\state')
\cdot \valueof{\mdp,B_{i+1}}{\state'} && \text{by \cref{inbody:eq:Ri-eq-Bi}}\\
& =\ \sum_{\state' \in F_i}
\probm_{\mdp,\state,\zstrat'}(R_{\le i} \cap \firstinset{F_i}=\state')
\cdot \expectval_{\mdp,\state',\zstrat_{i+1}}(B_{i+1}) && \text{opt.\ of $\zstrat_{i+1}$ for $B_{i+1}$}\\
& =\ \expectval_{\mdp,\state,\zstrat'}(B_{i+1}') && \text{by def.\ of
  $\zstrat'$ and $B_{i+1}'$}
\end{align*}
So $\zstrat'$ attains the value $\valueof{\mdp,B_{i+1}'}{\state}$ of the
objective $B_{i+1}'$ from $\state$ and is optimal. Thus \cref{inbody:eq:opt-Bi-prime}.

Now we show that $\zstrat'$ performs well on the objectives $R_{\le i}$ for
all $i \in \N$.
\begin{equation}\label{inbody:eq:1-bit-val}
\probm_{\mdp,\state,\zstrat'}(R_{\le i}) \ge \valueof{\mdp,\formula'}{\state}
\end{equation}
We have
\begin{align*}
\probm_{\mdp,\state,\zstrat'}(R_{\le i})\ &
\ge\ \expectval_{\mdp,\state,\zstrat'}(B_i') && \text{since $B_i'$ gives rewards
  $0$ for runs $\notin R_{\le i}$ and $\le 1$ otherwise} \\
  & =\ \valueof{\mdp,B_i'}{\state} && \text{by \cref{inbody:eq:opt-Bi-prime}}\\
  & = \  \valueof{\mdp,\formula'}{\state} && \text{by \cref{inbody:eq:Bi-prime}}
\end{align*}
So we get \cref{inbody:eq:1-bit-val}. Now we are ready to prove the optimality of $\zstrat'$ for $\formula'$ from $\state$.
\begin{align*}
\probm_{\mdp,\state,\zstrat'}(\formula')\ &
  =\ \probm_{\mdp,\state,\zstrat'}(\cap_{i \in \N} R_{\le i}) && \text{by def.\ of $\formula'$}\\
  & = \  \lim_{i \to \infty}\probm_{\mdp,\state,\zstrat'}(R_{\le i}) &&
\text{by continuity of measures from above}\\
  & \ge \ \lim_{i \to \infty}\valueof{\mdp,\formula'}{\state} && \text{by \cref{inbody:eq:1-bit-val}}\\
  & = \ \valueof{\mdp,\formula'}{\state}
\end{align*}

\medskip
\noindent
\smallparg{From finitely to infinitely branching MDPs.}
Encode infinite branching into finite branching like
in \cref{fig:infmem-Buchi}, apply the above result to obtain a 1-bit strategy
for the finitely branching version,
and then transform this strategy back into a 1-bit strategy for the original
MDP.
\end{proof}

Now we show our upper bound on the strategy complexity of \Buchi\ for general
MDPs.

\begin{restatable}{ourtheorem}{thmMarkovplusone}
For every countable MDP $\mdp$, finite set of initial states $\initstates$, set of
states $F$ and $\eps>0$, there exists a
deterministic 1-bit Markov strategy for $\reset(F)$
that is $\eps$-optimal
from every $\state \in \initstates$.
\end{restatable}
\begin{proof}
Encode a step-counter into the states to obtain an acyclic MDP, apply
\cref{thm:MDP-one-bit-Buchi}
to obtain an $\eps$-optimal deterministic 1-bit strategy for it,
and then transform this strategy back into
an $\eps$-optimal deterministic 1-bit Markov strategy in the original MDP.
\end{proof}


\bibliography{single-bib-file}

\newpage
\appendix
\section{The Lower Bound: Full Details}\label{app:lower}
\thmnoMarkov*

We follow the proof sketch from the main body and first argue that, in the MDP~$\M$ from \autoref{fig-no-MR}, no MR-strategy achieves a positive probability for the objective of visiting $F$ infinitely often and taking no red transition.
Indeed, given an MR-strategy and a tree, we define two probabilities:
\begin{itemize}
\item $s$ (for ``survival''): the probability that, starting in the top-left brown state, no red transition in the tree is visited;
\item $t$ (for ``total success''): the probability that, starting in the top-left brown state, at least one green state but no red transition in the tree is visited.
\end{itemize}
Trivially, $t \le s$.
A key lemma 
is the following.
\begin{ourlemma} \label{lem-no-MR-key}
Write $p \defeq 0.7$.
For every MR-strategy~$\sigma$ and every $n \in \N$, the tree $T^n$ satisfies:
\[
s \ \le \ a^{q t n^2}\;,
\]
where $a = 1 - \frac1{n^2 + 1}$ and $q = \frac19 (1-p)$.
\end{ourlemma}
\begin{proof}
Fix any MR-strategy~$\sigma$ and any $n \in \N$.
For each $k \in \{0, \ldots, n\}$, the tree~$T^n$ has $2^{n-k}$ height-$k$ subtrees, for which we can define $s, t$ analogously.
We claim: for all $k \in \{0, \ldots, n\}$ the probabilities $s, t$ in every height-$k$ subtree of~$T^n$ satisfy
\begin{equation} \label{eq-lem-no-MR-key-induction}
 s \le a^{(q t + \frac12 q t^2) k^2}\;,
\end{equation}
where $a = 1 - \frac1{n^2 + 1}$ and $q = \frac19 (1-p)$.
Note that the claim (for $k=n$) implies the lemma.

We prove the claim by induction on~$k$.
For the base case, $k=0$, note that each height-$0$ subtree of~$T^n$ consists of only a single green state.
Hence $s = t = 1$, so the claim holds for $k=0$.
For the inductive step, let $k \in \{1, \ldots, n\}$ and consider a height-$k$ subtree, say~$T$, of~$T^n$.
Let $T_0, T_1$ be the left and the right subtree of~$T$, respectively; they have height $k-1$.
In the two (yellow) topmost controlled states in~$T$, the MR-strategy~$\sigma$ chooses probabilities to visit $T_0, T_1$, respectively.
Taking into account the two brown random states at the top, the probabilities to visit $T_0, T_1$ are $p_0, p_1 \le p$, respectively.
In $T_0, T_1$, the strategy~$\sigma$ employs MR-strategies that achieve probabilities $s_0, t_0$ and $s_1, t_1$, respectively, where $s_i, t_i$ are defined in the obvious way for $T_i$.
By the induction hypothesis we have
\begin{equation} \label{eq-lem-no-MR-key-IH}
s_i \ \le \ a^{q t_i (1 + \frac12 t_i) (k-1)^2} \qquad \text{for $i \in \{0,1\}$.}
\end{equation}
By the structure of the MDP~$\M$ we have:
\begin{align}
s \ &= \ (1 - p_0 + p_0 a s_0) (1 - p_1 + p_1 a s_1) \label{eq-lem-no-MR-key-s-recursive} \\
t \ &= \ p_0 a t_0 (1 - p_1 + p_1 a s_1) + p_1 a t_1 (1 - p_0 + p_0 a s_0) - p_0 a t_0 p_1 a t_1  \nonumber \\
    &\le \ p_0 a t_0 + p_1 a t_1 - p_0 a t_0 p_1 a t_1 \label{eq-lem-no-MR-key-t-recursive}
\end{align}
By combining \Cref{eq-lem-no-MR-key-IH,eq-lem-no-MR-key-s-recursive} we obtain:
\begin{equation} \label{eq-lem-no-MR-key-lhs}
s \ \le \ \prod_{i=0}^1 \left(1 - p_i + p_i a^{1 + q t_i (1 + \frac12 t_i) (k^2-2k)}\right) 
\end{equation}
On the other hand, from~\cref{eq-lem-no-MR-key-t-recursive} we obtain:
\begin{equation} \label{eq-lem-no-MR-key-rhs1}
\begin{aligned}
q t + \frac12 q t^2 \ &\le \ q \left( p_0 t_0 + p_1 t_1 - p_0 a t_0 p_1 a t_1 \right) + \frac12 q \left( p_0 a t_0 + p_1 a t_1 \right)^2 \\
&\le \ p_0 q t_0 \left(1 + \frac12 p_0 t_0 \right) + p_1 q t_1 \left(1 + \frac12 p_1 t_1 \right)
\end{aligned}
\end{equation}
Further we have:
\begin{equation} \label{eq-lem-no-MR-key-ln1a}
\ln \frac{1}{a} \ \le \ \frac1{a} - 1 \ = \ \frac{n^2 + 1}{n^2} - 1 \ = \ \frac{1}{n^2} \ \le \ \frac{1}{k^2}
\end{equation}
Let $i \in \{0,1\}$.
By~\cref{eq-lem-no-MR-key-ln1a} we have:
\begin{equation*} \label{eq-lem-no-MR-key-range}
\left(\ln \frac1{a}\right) p_i q t_i \left(1 + \frac12 p_i t_i \right) k^2 \ \le \ q \left(1 + \frac12\right) \ \le \ \frac19 \cdot \frac32 \ < \ \frac12
\end{equation*}
Hence, using a bound on the exponential function (\autoref{lem-linear-exponential} below), 
we obtain:
\begin{align*}
a^{p_i q t_i \left(1 + \frac12 p_i t_i \right) k^2} \
&= \ e^{- p_i (\ln \frac1{a}) q t_i \left(1 + \frac12 p_i t_i \right) k^2} \\
& \ge \ 1 - p_i + p_i e^{-(\ln \frac1{a}) q t_i \left(1 + \frac12 p_i t_i \right) k^2 - (\ln \frac1{a})^2 q^2 t_i^2 \left(1 + \frac12 p_i t_i \right)^2 k^4} \\
& \mathop{\ge}^\text{\cref{eq-lem-no-MR-key-ln1a}} \ 1 - p_i + p_i e^{-(\ln \frac1{a}) q t_i \left(1 + \frac12 p_i t_i \right) k^2 - (\ln \frac1{a}) \frac94 q^2 t_i^2 k^2} \\
&= \ 1 - p_i + p_i a^{q t_i k^2 + q \left(\frac12 p_i + \frac94 q \right) t_i^2 k^2}
\end{align*}
By combining this inequality with \cref{eq-lem-no-MR-key-rhs1} we obtain:
\[
a^{\left(q t + \frac12 q t^2\right) k^2} \ \ge \ \prod_{i=0}^1 \left(1 - p_i + p_i a^{q t_i k^2 + q \left(\frac12 p_i + \frac94 q \right) t_i^2 k^2}\right) 
\]
Considering~\cref{eq-lem-no-MR-key-lhs}, we see that, in order to prove~\cref{eq-lem-no-MR-key-induction}, it suffices to prove
\begin{align*}
1 + q t_i \left(1 + \frac12 t_i\right) (k^2-2k) \ &\ge \ q t_i k^2 + q \left(\frac12 p_i + \frac94 q \right) t_i^2 k^2 \qquad \text{for }i \in \{0,1\}. \label{eq-lem-no-MR-key-suffices1}
\end{align*}
This inequality 
is equivalent to:
\begin{alignat*}{2}
&  & 1 + q t_i k \left( \left( \frac12 - \frac12 p_i - \frac94 q \right) t_i k - 2 \left( 1 + \frac12 t_i \right) \right)  \ &\ge \ 0 \\
&\Longleftarrow\quad& 1 + q t_i k \left( \left( \frac12 (1 - p) - \frac94 q \right) t_i k - 3 \right) \ &\ge \ 0 \\
&\Longleftrightarrow\quad& 1 + \frac19 (1-p) t_i k \left( \frac14 (1 - p) t_i k - 3 \right) \ &\ge \ 0 \\
&\Longleftrightarrow\quad&  \left( \frac16 (1-p) t_i k - 1 \right)^2  \ &\ge \ 0
\end{alignat*}
The left-hand side is a square, hence nonnegative.
This completes the induction proof.
\end{proof}

The following elementary lemma from calculus was used in the proof of \autoref{lem-no-MR-key}.
\begin{ourlemma} \label{lem-linear-exponential}
For every $r \ge 0$ and $x \in [0, \frac12]$ we have $e^{- r x} \ge 1 - r + r e^{-x -x^2}$.
\end{ourlemma}
\begin{proof}
Let $r \ge 0$ and $x \in [0, \frac12]$.
As $1+y \le e^y$ holds for all~$y$, we have:
\[
1 - r + r e^{-x -x^2} \ = \ 1 - r \left(1 - e^{-x - x^2}\right) \ \le \ e^{-r \left(1 - e^{-x-x^2}\right)}
\]
Hence it suffices to prove that $x \le 1 - e^{-x - x^2}$, which is equivalent to $\ln(1-x) + x + x^2 \ge 0$.
To prove the latter inequality, define $f(y) \defeq \ln(1-y) + y + y^2$.
Then we have $f(0) = 0$ and
\[
f'(y) \ = \ -\frac1{1-y} + 1 + 2 y \ = \ \frac{-1 + 1 - y + 2y - 2y^2}{1-y} \ = \ \frac{y (1-2 y)}{1-y} \ \ge \ 0 \ \text{ for } y \in \left[0, \frac12\right]. 
\]
By the fundamental theorem of calculus, it follows $f(x) = f(0) + \int_0^x f'(y) \; d y \ge 0$.
\end{proof}

\begin{restatable}{ourlemma}{thmnoMR} \label{thm-no-MR}
Consider the acyclic MDP~$\M$ shown in \autoref{fig-no-MR}.
Let $\formula$ be the objective of visiting infinitely many green states and no red transition.
\begin{enumerate}
\item
For every MR-strategy~$\zstrat$, we have $\probm_{\mdp,\state_0,\zstrat}(\formula) = 0$.
\item
$\valueof{\formula}{\state_0} = 1$ and for every $\eps>0$
there exists a deterministic 1-bit strategy~$\zstrat_\eps$ s.t.\ $\probm_{\mdp,\state_0,\zstrat_\eps}(\formula) \ge 1-\eps$.
\end{enumerate}
\end{restatable}

\begin{proof}
First we prove item~1.
Fix any MR-strategy~$\sigma$.
For each $n \in \N$, let $s_n, t_n$ denote the probabilities $s,t$ for the tree~$T^n$ under~$\sigma$.
Define also $d_n \defeq 1 - s_n$ (for ``death''), which is the probability of taking at least one red transition starting in the top-left brown state of~$T^n$.
For the following estimate, observe that we have
\begin{equation} \label{eq-thm-no-MR-euler}
\left(1 - \frac1{x+1}\right)^x \ = \ e^{x \ln\left(1 - \frac1{x+1}\right)} \ \le \ e^{-\frac{x}{x+1}} \ \le \ e^{-\frac12} \qquad \text{for $x \ge 1$.}
\end{equation}
By \autoref{lem-no-MR-key} we have for every~$n$:
\begin{equation} \label{eq-thm-no-MR-death}
d_n \ = \ 1 - s_n \ \ge \ 1 - \left( 1 - \frac1{n^2 + 1} \right)^{q t_n n^2} \ \mathop{\ge}^\text{\cref{eq-thm-no-MR-euler}} \ 1 - e^{-\frac12 q t_n} \ \ge \ \frac14 q t_n \;,
\end{equation}
where the last inequality follows from the fact that $e^{-x} \le 1 - \frac12 x$ holds for $x \in [0, 1]$.

Denote by~$G_n$ the indicator random variable such that
\begin{itemize}
\item $G_n = 1$ if the top-left brown state of~$T^n$ is visited (coming from the previous blue state) and at least one green state in~$T^n$ but no red transition in~$T^n$ is visited;
\item $G_n = 0$ otherwise.
\end{itemize}
Considering that the probability of visiting the top-left brown state of~$T^n$ is $\frac1{n}$, we have ${\expectval} G_n = \frac1{n} \cdot t_n$, where ${\expectval}$ denotes expectation.

If $\sigma$ visits at least one red transition in~$\M$ almost surely then the probability of~$\formula$ is~$0$.
Therefore, suppose $\sigma$ achieves a positive probability, $\bar{r} > 0$, of visiting no red transition.
Since $0 < \bar{r} =  \prod_{n=1}^\infty \left(1 - \frac1{n} \cdot d_n\right)$, the series $\sum_{n=1}^\infty \frac1{n} \cdot d_n$ converges.
Thus:
\begin{align*}
  \expectval \sum_{n=1}^\infty G_n \ = \ \sum_{n=1}^\infty \expectval G_n \ = \ \sum_{n=1}^\infty \frac1{n} \cdot t_n \ \mathop{\le}^\text{\cref{eq-thm-no-MR-death}} \ \frac{4}{q} \cdot \sum_{n=1}^\infty \frac1{n} \cdot d_n \ < \ \infty
\end{align*}
It follows that the probability that $\sum_{n=1}^\infty G_n$ diverges is~$0$.
But on~$\formula$ the series $\sum_{n=1}^\infty G_n$ diverges.
Hence the probability of~$\formula$ is~$0$.
This completes the proof of item~1.

Towards item~2, we first define a suitable strategy, $\sigma$,
that achieves a positive value (i.e.,
$\probm_{\mdp,\state_0,\sigma}(\formula)> 0$)
and then improve it to obtain $\eps$-optimal strategies $\zstrat_\eps$.

The strategy~$\sigma$ acts independently in each tree~$T^n$.
In each tree $T^n$ the strategy~$\sigma$ maximizes the probability of visiting exactly one green state.
To this end, as long as $\sigma$ has not yet visited a green state in~$T^n$, it chooses the downward transition emanating from the yellow controlled states; as soon as a green state in~$T^n$ has been visited, $\sigma$ chooses the upward transition emanating from the yellow controlled states, thus avoiding any further visit of a green state or a red transition in~$T_n$.
This is a 1-bit strategy, as $\sigma$ remembers only whether a green state has already been visited in the current tree~$T^n$.
The bit is reset whenever a new tree is entered.

Next we show that $\sigma$ visits infinitely many green states with probability~$1$.
Let $u_n$ denote the probability that, starting in the top-left brown state of~$T^n$, no green state is visited in~$T^n$.
Define $u_0 = 0$.
Since red or non-red transitions are unimportant for the current considerations, all height-$n$ trees in~$\M$ have the same structure, even when they are subtrees of different~$T^m$.
Therefore we have:
\[
 u_n \ = \ (p u_{n-1} + 1 - p)^2
\]
Since the function $f(x) \defeq (p x + 1 - p)^2$ is monotone on $[0,1]$, the sequence $(u_n)_n$ is nondecreasing and thus converges to the smaller fixed point, $u$, of~$f$.
Hence,
\begin{equation} \label{eq-u-fixed-point}
0 \ \le \ u_n \ \le \ u \ = \ f(u) \ = \ \left(\frac{1-p}{p}\right)^2 \ < \ 1 \qquad \text{for all }n \in \N \cup \{0\}.
\end{equation}
It follows that we have
\begin{equation} \label{eq-inf-green}
 \sum_{n=k}^\infty \frac{1}{n} (1-u_n) \ \ge \ (1-u) \sum_{n=k}^\infty \frac1{n} \ = \ \infty \qquad \text{for all $k \in \N$.}
\end{equation}
For every $k \in \N$, the probability that, starting in the blue state directly before~$T^k$, no green state in $T^k, T^{k+1}, \ldots$ is visited is
\[
\prod_{n=k}^\infty \left( \frac1{n} \cdot u_n + \left(1 - \frac1{n}\right) \right) \ = \ \prod_{n=k}^\infty \left(1 - \frac{1}{n} (1-u_n) \right) \ \mathop{=}^\text{by \cref{eq-inf-green}} \ 0\,.
\]
It follows that $\sigma$ visits infinitely many green states with probability~$1$.

It now suffices to show that, with positive probability, $\sigma$ visits no red transition.
Let $v_n$ denote the expectation, starting in the top-left brown state of~$T^n$, of the number of red \emph{states} (not red \emph{transitions}) that are visited in~$T^n$.
Define $v_0 \defeq 0$.
Since red or non-red transitions are unimportant for the current considerations, all height-$n$ trees in~$\M$ have the same structure, even when they are subtrees of different~$T^m$.
Therefore we have:
\begin{align} \label{eq-red-states-vn}
 v_n \ &= \ p \left(1 + v_{n-1} + u_{n-1} p (1 + v_{n-1}) \right) \ + \ (1-p) p (1 + v_{n-1})
\end{align}
We prove by induction that $v_n \le n$ holds for all $n \in \N \cup \{0\}$.
The base case, $n=0$, holds by the definition of~$v_0$.
For the inductive step, let $n \ge 1$.
We have:
\begin{align*}
 v_n \ &\le \ p (n + u_{n-1} p n) + (1-p) p n && \text{by~\cref{eq-red-states-vn} and the induction hypothesis}\\
       &\le \ p \left( n + \frac{(1-p)^2}{p} n\right) + (1-p) p n && \text{by~\cref{eq-u-fixed-point}} \\
       &= \ p n + (1-p)^2 n + p n - p^2 n \ = \ n
\end{align*}
Hence we have proved $v_n \le n$.
It follows that the expectation, starting in the top-left brown state of~$T^n$, of the number of red \emph{transitions} visited in~$T^n$ is at most $n \cdot \frac{1}{n^2 + 1}$.
Thus the expected number of visited red transitions in the whole MDP~$\M$ is at most $\sum_{n=1}^\infty \frac1{n} \cdot n \cdot \frac{1}{n^2 + 1} \le \frac{\pi^2}{6}$.
Hence there is $k \in \N$ such that the expected number of red transitions visited in $T^k, T^{k+1}, \ldots$ is less than~$1$.
It follows from the Markov inequality that the probability to visit at least one red transition in $T^k, T^{k+1}, \ldots$ is less than~$1$.
Hence the probability to visit at least one red transition in~$\M$ is less
than~$1$.

The strategy~$\sigma$ from above can be improved
to obtain an $\eps$-optimal strategy $\zstrat_\eps$ for \Buchi\ from
$\state_0$,
i.e., $\probm_{\mdp,\state_0,\zstrat_\eps}(\formula) \ge 1-\eps$.
We obtain $\zstrat_\eps$ by modifying the described strategy~$\sigma$ such
that, in the first $k$ trees for some $k \in \N$, the upward transitions emanating from the yellow states are taken.
By choosing a large but finite~$k$, the risk of taking a red transition can be made arbitrarily small, while the probability of visiting infinitely many green states remains~$1$.
\end{proof}

Finally, we are ready to prove our main claim, \cref{thm-no-Markov}.
\begin{proof}[Proof of \cref{thm-no-Markov}]
We describe how to modify the MDP~$\M$ from \autoref{thm-no-MR} to obtain an MDP~$\M_2$ with the claimed properties.
First eliminate the red transitions in~$\M$ and change the objective to the normal \Buchi\ objective.
This can be done by redirecting all red transitions to an infinite (losing) chain of non-green states.
Denote the resulting MDP by $\M_1$.
For a state~$s$, define its \emph{depth}~$d(s)$ as the length of the \emph{longest} path from the start state $s_0$ to~$s$.
In~$\M_1$ each state has finite depth (this property does not follow from acyclicity alone).
Now obtain~$\M_2$ from $\M_1$ by replacing every transition that leads from a state $s_1$ to a state $s_2$ with $d(s_1) + 1 < d(s_2)$ by a chain (of non-green states) of length $d(s_2) - d(s_1)$.
In this way, in~$\M_2$, for every state~$s$, all paths from $s_0$ to~$s$ have the same length $d(s)$.
Thus, instrumenting $\M_2$ with a step-counter would lead to an MDP isomorphic to~$\M_2$.
It follows that every Markov strategy for~$\M_2$ could be replaced by an MR-strategy that achieves $\reset(F)$ with the same probability.
Observe that a MR-strategy for $M_2$ directly translates to an MR-strategy for $M$ that achieves the same probability.
Hence, item~$1$ follows, as the existence of a Markov-, and hence MR-strategy that achieves positive probability would contradict \autoref{thm-no-MR}.

Item~$2$ is shown by modifying the strategies~$\zstrat_\eps$ from item~$2$
of \autoref{thm-no-MR} in the natural way.
\end{proof}

A slight modification of the example above yields a lower bound on the memory
requirements for the almost-sure parity objective.
Recall that the parity objective is defined on systems whose states are
labeled by a finite set of colors $C \eqdef \{1,2,\dots,{\it max}\} \subseteq \N$,
where a run is in $\Parity{C}$ iff the highest color that is seen infinitely
often in the run is even.

\corParity*
\begin{proof}
We obtain $\mdp'$ by modifying the MDP $\mdp$ from
\autoref{thm-no-Markov} as follows.
Label all green states in $F$ by color $2$ and the rest by color $1$.
Then modify each red transition to go to its target via a fresh state labeled by color $3$.
Clearly $\mdp'$ is still acyclic and labeled by colors $\{1,2,3\}$.

From the proof of \autoref{thm-no-Markov} (1), under every Markov strategy in
$\mdp$ a.s.\ seeing infinitely many green states (in $F$) implies seeing infinitely
many red transitions.
So in $\mdp'$ every Markov strategy $\zstrat$ a.s.\ either sees color $2$ only finitely often or
color $3$ infinitely often, thus $\probm_{\mdp',\state_0,\zstrat}(\Parity{\{1,2,3\}}) = 0$.

From the proof of \autoref{thm-no-Markov} (2), there is a deterministic 1-bit strategy
$\zstrat$ in $\mdp$ that attains probability $\ge 1/2$ for $\reset(F)$
without taking any red transition
and otherwise a.s.\ takes a red transition.
This property of $\zstrat$ holds not only when starting from $\state_0$ but from every other
state as well. We obtain $\zstrat'$ in $\mdp'$ by
continuing to play $\zstrat$ even after red transitions have been taken.
Under $\zstrat'$ the probability of going through infinitely many red transitions
(and seeing color 3) is $\le (1/2)^\infty =0$, and the probability of seeing
infinitely many states in $F$ (with color $2$) is $1$.
Thus $\probm_{\mdp',\state_0,\zstrat'}(\Parity{\{1,2,3\}}) = 1$.
\end{proof}
\section{The Upper Bound: Full Details}\label{app:upper}
\thmMDPonebitBuchi*
\begin{proof}
  Let $\mdp=\mdptuple$ be an acyclic MDP,
  $\initstates \subseteq \states$ a finite set of initial states and $F \subseteq \states$ a 
  set of goal states and $\formula \eqdef \reset(F)$
  denote the \Buchi\ objective w.r.t.\ $F$.
  We prove the claim for finitely branching $\mdp$ first and transfer the result to general MDPs at the end.

  \renewcommand{\optval}[1][\formula]{\valueof{\mdp,#1}{\state}}

For every $\eps >0$ and every $\state \in \initstates$ there exists an $\eps$-optimal strategy $\zstrat_\state$
such that
\begin{equation}\label{eq:eps-opt}
\probm_{\mdp,\state,\zstrat_\state}(\formula) \ge \optval-\eps.
\end{equation}
However, the strategies $\zstrat_\state$ might differ from each other and
might use randomization and a large (or even infinite)
amount of memory.
We will construct a single deterministic strategy $\zstrat'$ that uses only 1 bit of
memory such that $\forall_{\state \in \initstates}\, \probm_{\mdp,\state,\zstrat'}(\formula) \ge \optval-2\eps$.
This proves the claim as $\eps$ can be chosen arbitrarily small.

In order to construct $\zstrat'$, we first observe the behavior of the
finitely many $\zstrat_\state$ for $\state \in \initstates$
on an infinite, increasing sequence of finite subsets of $\states$.
Based on this, we define a second stronger objective $\formula'$ with
\begin{equation}\label{eq:prime-implies-normal}
\formula' \subseteq \formula,
\end{equation}
and show that all $\zstrat_\state$ 
attain at least $\optval-2\eps$
w.r.t.\ $\formula'$, i.e.,
\begin{equation}\label{eq:observe-orig}
    \forall_{\state \in \initstates}\, \probm_{\mdp,\state,\zstrat_\state}(\formula') \ge \optval-2\eps.
\end{equation}
We construct $\zstrat'$ as a deterministic 1-bit \emph{optimal} strategy w.r.t.\ $\formula'$
from all $\state \in \initstates$ and obtain
\begin{align*}
\probm_{\mdp,\state,\zstrat'}(\formula)\ &
\ge\ \probm_{\mdp,\state,\zstrat'}(\formula') && \text{by \cref{eq:prime-implies-normal}} \\
  & \ge\ \probm_{\mdp,\state,\zstrat_\state}(\formula') && \text{by optimality of $\zstrat'$ for $\formula'$}\\
  & \ge \ \optval-2\eps && \text{by \cref{eq:observe-orig}}.
\end{align*}

\medskip
\noindent
\smallparg{Behavior of $\zstrat$, objective $\formula'$ and properties
  \cref{eq:prime-implies-normal} and \cref{eq:observe-orig}.}
We start with some notation.
Let $\bubble{k}{X}$ be the set of states that can be reached from some state
in the set $X$ within at most $k$ steps.
Since $\mdp$ is finitely branching, $\bubble{k}{X}$ is finite if $X$ is
finite.
Let
$\setfb{\le k}{X} \eqdef\{\play \in \states^{\omega} \mid \exists t \le k.\, X(\play(t))=1\}$
and
$\setfb{\ge k}{X} \eqdef\{\play \in \states^{\omega} \mid \exists t \ge k.\, X(\play(t))=1\}$
denote the property of visiting the set $X$ (at least once) within at most
(resp.\ at least) $k$ steps.
Moreover, let $\eps_i \eqdef \eps \, \cdot\, 2^{-(i+1)}$.

\begin{figure*}[t]
   \begin{center}			
\centering

\begin{tikzpicture}[>=latex',shorten >=1pt,node distance=1.9cm,on grid,auto,
roundnode/.style={circle, draw,minimum size=1.5mm},
squarenode/.style={rectangle, draw,minimum size=2mm}]

\draw (3.8,0) ellipse (4.1cm and 1.3cm);
\draw [green!50!black,fill=green!40!white](-.3,0) arc (180:0:4.1cm and 1.3cm);
\draw [fill=white](2.8,0) ellipse (3.1cm and 1.1cm);
\draw (2.3,0) ellipse (2.6cm and .9cm);
\draw [green!50!black,fill=green!40!white](-.3,0) arc (180:0:2.6cm and .9cm);
\draw [fill=white](1.3,0) ellipse (1.6cm and .7cm);
\draw (.8,0) ellipse (1.1cm and .5cm);
\draw [green!50!black,fill=green!40!white](-.28,0) arc (180:0:1.09cm and .5cm);

\draw (11,1.39) arc (45:-45:3cm and 2cm);
\draw  (13.05,0.05) arc (0:-50:3cm and 2.2cm);

\draw [draw=none,name path=A] (13.05,0.05) arc (0:50:3cm and 2.2cm);
\draw [draw=none,name path=B](11,1.39) arc (45:0:3cm and 2cm);
\tikzfillbetween[of=A and B]{green!40!white};
\draw [green!50!black,name path=A] (13.05,0.05) arc (0:50:3cm and 2.2cm);
\draw [green!50!black,name path=B](11,1.39) arc (45:0:3cm and 2cm);

\draw (10,1.2) arc (40:-40:3cm and 1.8cm);
\draw[-,green!50!black] (0,0)--(1.95,0.05);
\draw[-,green!50!black] (2.9,0.02)--(4.95,0.05);
\draw[-,green!50!black](5.9,0.02)--(7.95,0.05);
\draw[-,green!50!black](11.87,-0.01)--(13.1,0.05);

\node [roundnode,fill=white] (s) at (0,0) {$I$};

\node[draw=none](dot1) at (9.2,0) {{\large $\cdots$}};
\node[draw=none](dot2) at (14,0) {{\large $\cdots$}};

\node[draw=none](K1) at (1.15,-.3)  {$K_1$};
\node[draw=none](L1) at (2.3,-.3)   {$L_1$};
\node[draw=none](K2) at (4.3,-.3)   {$K_2$};
\node[draw=none](L2) at (5.4,-.3)   {$L_2$};
\node[draw=none](K3) at (7.4,-.3)   {$K_3$};
\node[draw=none](Li) at (11.4,-.3)  {$L_{i-1}$};
\node[draw=none](Kii) at (12.6,-.3) {$K_i$};
\node[draw=none](Lii) at (13.6,-.3) {$L_i$};

\end{tikzpicture}
		\end{center}
	\caption{To show the bubble construction. The green region in $K_1$ is $F_1$, and for all $i\geq 2$, the green region in
	$K_{i}\setminus L_{i-1}$ is $F_i$.
}
\label{fig:KLBubbles}
\end{figure*}
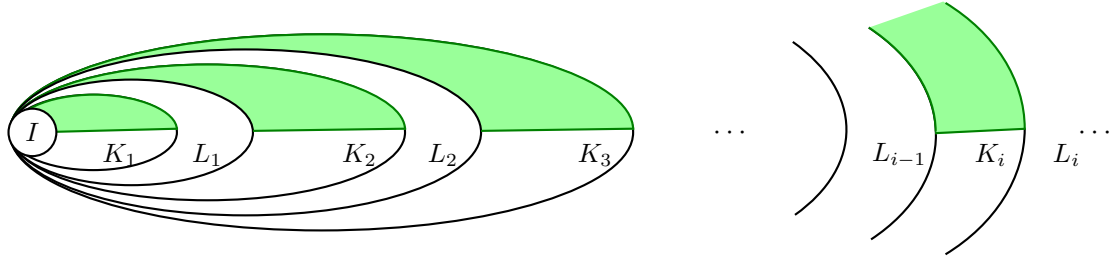

The following lemma depends on the assumption that $\mdp$ is acyclic. 

\begin{ourlemma}\label{lem-bubble-extension}
Let $X \subseteq \states$ be a finite set of states and $\eps' > 0$.
\begin{enumerate}
\item
There is $k \in \N$ such that
$
\forall_{\state\in\initstates}\,\probm_{\mdp,\state,\zstrat_\state}(\formula \cap \complementof{\setfb{\le k}{F \setminus X}}) \ \le \ \eps'.
$\label{lem-bubble-extension:ad1}
\item
There is $l \in \N$ such that $\forall_{\state\in\initstates}\,\probm_{\mdp,\state,\zstrat_\state}(\setfb{\ge l}{X}) \ \le \ \eps'$.
\label{lem-bubble-extension:ad2}
\end{enumerate}
\end{ourlemma}
\begin{proof}
It suffices to show the properties for a single $\state,\zstrat_\state$ since
one can take the maximal $k,l$ over the finitely many $\state \in \initstates$.
By acyclicity of $\mdp$, it holds that $\formula ~\subseteq~ \setf(F \setminus X) = \bigcup_{k \in \N} \setfb{\le k}{F \setminus X}$
 and therefore that $\formula \cap \bigcap_{k \in \N} \complementof{ \setfb{\le k}{F \setminus X}} = \emptyset$.
 It follows from the continuity of measures that $\lim_{k \to \infty} \probm_{\mdp,\state,\zstrat_\state}(\formula \cap \complementof{\setfb{\le k}{F \setminus X}}) = 0$.

Item~2 follows directly from the fact that $\mdp$ is acyclic.
\end{proof}

By Lemma~\ref{lem-bubble-extension}(\ref{lem-bubble-extension:ad1}) there is a
$k_1$ such that for $K_1 \eqdef \bubble{k_1}{\initstates}$ and $F_1 \eqdef F \cap K_1$ we
have $\forall_{\state\in\initstates}\,\probm_{\mdp,\state,\zstrat_\state}(\formula \cap \complementof{K_1^* F_1 S^\omega}) \le \eps_1$.
We define the pattern
$$R_1 \eqdef (K_1 \setminus F_1)^*F_1$$
and obtain $\forall_{\state\in\initstates}\,\probm_{\mdp,\state,\zstrat_\state}(\formula \cap \complementof{R_1 S^\omega}) \le \eps_1$.
By Lemma~\ref{lem-bubble-extension}(\ref{lem-bubble-extension:ad2})
there is an $l_1 > k_1$ such that $\forall_{\state\in\initstates}\,\probm_{\mdp,\state,\zstrat_\state}(\setfb{\ge l_1}{K_1}) \le \eps_1$.
Define $L_1 \eqdef \bubble{l_1}{\initstates}$.
By Lemma~\ref{lem-bubble-extension}(\ref{lem-bubble-extension:ad1}) there is a
$k_2 > l_1$ such that for $K_2 \eqdef \bubble{k_2}{\initstates}$ and
$F_2 \eqdef F \cap K_2 \setminus L_1$ we have $\forall_{\state\in\initstates}\,\probm_{\mdp,\state,\zstrat_\state}(\formula \cap \complementof{K_2^* F_2 S^\omega}) \le \eps_2$.
We define the pattern
$$R_2 \eqdef (K_2 \setminus F_2)^*F_2$$
and obtain
$\forall_{\state\in\initstates}\,\probm_{\mdp,\state,\zstrat_\state}(\formula \cap \complementof{R_2
S^\omega}) \le \eps_2$ and, via a union bound,
$\forall_{\state\in\initstates}\,\probm_{\mdp,\state,\zstrat_\state}(\formula \cap \complementof{R_2 (S \setminus K_1)^\omega}) \le \eps_1 + \eps_2$.
By another union bound it follows that $\forall_{\state\in\initstates}\,\probm_{\mdp,\state,\zstrat_\state}(\formula \cap \complementof{R_1 R_2 (S \setminus K_1)^\omega}) \le 2 \eps_1 + \eps_2$.

Proceed inductively for $i = 2, 3, \ldots$ as follows (see \cref{fig:KLBubbles} for an illustration).
By Lemma~\ref{lem-bubble-extension}(\ref{lem-bubble-extension:ad2})
there is an $l_i > k_i$ such that $\forall_{\state\in\initstates}\,\probm_{\mdp,\state,\zstrat_\state}(\setfb{\ge l_i}{K_i}) \le \eps_i$.
Define $L_i \eqdef \bubble{l_i}{\initstates}$.
By Lemma~\ref{lem-bubble-extension}(\ref{lem-bubble-extension:ad1}) there is
$k_{i+1} > l_i$ such that for
$K_{i+1} \eqdef \bubble{k_{i+1}}{\initstates}$ and $F_{i+1} \eqdef F \cap
K_{i+1} \setminus L_i$ we have
$\forall_{\state\in\initstates}\,\probm_{\mdp,\state,\zstrat_\state}(\formula \cap \complementof{(K_{i+1} \setminus F_{i+1})^* F_{i+1} S^\omega}) \le \eps_{i+1}$.
By a union bound, $\forall_{\state\in\initstates}\,\probm_{\mdp,\state,\zstrat_\state}(\formula \cap \complementof{(K_{i+1} \setminus F_{i+1})^* F_{i+1} (S \setminus K_i)^\omega}) \le \eps_i + \eps_{i+1}$.
By an induction hypothesis we have $\forall_{\state\in\initstates}\,\probm_{\mdp,\state,\zstrat_\state}(\formula \cap \complementof{R_1 R_2 \ldots R_i (S \setminus K_{i-1})^\omega}) \le 2 \eps_1 + \cdots + 2 \eps_{i-1} + \eps_i$.
We define the pattern
$$R_{i+1} \eqdef (K_{i+1} \setminus (F_{i+1} \cup K_{i-1}))^*F_{i+1}.$$
Using that
$
 (K_{i+1} \setminus F_{i+1})^* F_{i+1} (S \setminus K_i)^\omega \ \cap \ R_1 R_2 \ldots R_i (S \setminus K_{i-1})^\omega
\subseteq
R_1 R_2 \ldots R_{i+1} (S \setminus K_i)^\omega,
 $
we get
\begin{equation}\label{eq:eps-bound}
    \forall_{\state\in\initstates}\,\probm_{\mdp,\state,\zstrat_\state}(\formula \cap \complementof{R_1 R_2 \ldots R_{i+1} (S \setminus K_i)^\omega})
\quad\le\quad
2 \eps_1 + \cdots + 2 \eps_{i} + \eps_{i+1}
\quad\le\quad
\eps.
\end{equation}

We now define the Borel objectives
$R_{\le i} \eqdef R_1 R_2 \dots R_i \states^\omega$ and
$\formula' \eqdef \bigcap_{i \in \N} R_{\le i}$.
Since $F_i \cap F_k = \emptyset$ for $i \neq k$ and $\formula'$ implies a visit
to the set $F_i$ for all $i \in \N$, we have
$\formula' \subseteq \formula$ and obtain \cref{eq:prime-implies-normal}.
Moreover, $R_{\le 1} \supseteq R_{\le 2} \supseteq R_{\le 3} \dots$ is an
infinite decreasing sequence of Borel objectives.
For every $\state \in \initstates$ we have
\begin{align*}
\probm_{\mdp,\state,\zstrat_\state}(\formula')\ &
=\ \probm_{\mdp,\state,\zstrat_\state}(\cap_{i=1}^\infty R_{\le i}) && \text{by def.\ of $\formula'$}\\
& = \ \lim_{i \to \infty} \probm_{\mdp,\state,\zstrat_\state}(R_{\le i}) && \text{by cont.\ of measures}\\
& = \ \lim_{i \to \infty} 1-\probm_{\mdp,\state,\zstrat_\state}(\complementof{R_{\le i}}) && \text{by duality}\\
& = \ \lim_{i \to \infty} 1-(\probm_{\mdp,\state,\zstrat_\state}(\complementof{R_{\le i}}\cap
\formula) + \probm_{\mdp,\state,\zstrat_\state}(\complementof{R_{\le i}} \cap \complementof{\formula})) && \text{case split}\\
                                                                                                 & \ge \ \lim_{i \to \infty} 1-(\eps + \probm_{\mdp,\state,\zstrat_\state}(\complementof{R_{\le i}} \cap  \complementof{\formula})) &&
\text{by \cref{eq:eps-bound}}\\
& \ge \ \lim_{i \to \infty} 1-(\eps + \probm_{\mdp,\state,\zstrat_\state}(\complementof{\formula'}\cap \complementof{\formula})) &&
\text{since $\formula' \subseteq R_{\le i}$}\\
& = \ 1-(\eps + 1-\probm_{\mdp,\state,\zstrat_\state}(\formula' \cup \formula)) &&
\text{by duality}\\
& = \ \probm_{\mdp,\state,\zstrat_\state}(\formula) - \eps &&
\text{by \cref{eq:prime-implies-normal}}\\
& \ge \ \optval-2\eps && \text{by \cref{eq:eps-opt}}
\end{align*}
Thus we obtain property \cref{eq:observe-orig}.

\medskip
\noindent
\smallparg{Definition of the 1-bit strategy $\zstrat'$.}
We now define our deterministic 1-bit strategy $\zstrat'$ that is optimal for
objective $\formula'$ from \emph{every} $\state \in \initstates$.
First we define certain ``suffix'' objectives of $\formula'$.
Recall that $R_i = (K_{i} \setminus (F_{i} \cup K_{i-2}))^*F_{i}$.
Let $R_{i,j} \eqdef R_i R_{i+1} \dots R_j \states^\omega$
and $R_{\ge i} \eqdef \bigcap_{j \ge i} R_{i,j}$.
In particular, this means that $\formula' = R_{\ge 1}$.
Every run $w$ from some state $\state \in \initstates$ that satisfies $\formula'$ can be split into parts
before and after the first visit to set $F_i$, i.e.,
$w = w_1\state'w_2$ where $w_1\state' \in R_{\le i}$, $\state' \in F_i$ and
$\state' w_2 \in R_{\ge i+1}$.
(Note also that $w_2$ cannot visit any states in $K_{i-1}$.)
Thus it will be useful to consider the
objectives $R_{\ge i+1}$ for runs that start in states $\state' \in F_i$.
For every state $\state' \in F_i$ we consider its value w.r.t.\ the
objective $R_{\ge i+1}$, i.e.,
$\valueof{\mdp,R_{\ge i+1}}{\state'} \eqdef \sup_{\hat{\zstrat}} \probm_{\mdp,\state',\hat{\zstrat}}(R_{\ge i+1})$.

For every $i \ge 1$ we consider the finite subspace $K_i\setminus K_{i-2}$.
In particular, it contains the sets $F_{i-1}$ and $F_i$.
(For completeness let $K_0 \eqdef F_0 \eqdef \initstates$ and
$K_{-1} \eqdef \emptyset$.)
It is not enough to maximize the probability of reaching the set $F_i$ in each
$K_i$ individually. One also needs to maximize the potential of visiting
further sets $F_{i+1}, F_{i+2}, \dots$ in the indefinite future.
Thus we define the bounded total reward
objective $B_i$ for runs starting in $F_{i-1}$ as follows.
Runs that exit the subspace (either by leaving $K_i$ or by visiting $K_{i-2}$)
before visiting $F_i$ get reward $0$. All other runs must visit $F_i$
eventually (since $\mdp$ is acyclic and the subspace is finite).
When some run reaches the set $F_i$ \emph{for the first time}
in some state $\state'$ then this run gets the reward of $\valueof{\mdp,R_{\ge i+1}}{\state'}$.
We can consider an induced finite MDP $\hat{\mdp}$ with state space $K_i\setminus K_{i-2}$,
plus a sink state (with reward $0$) that is reached immediately after visiting
any state in $F_i$ and whenever one exits the set $K_i\setminus K_{i-2}$.
In $\hat{\mdp}$ one gets a reward of
$\valueof{\mdp,R_{\ge i+1}}{\state'}$ for visiting $\state' \in F_i$ as above.
By \cite[Theorem 7.1.9]{Puterman:book}, there exists a uniform optimal MD-strategy $\zstrat_i$ for this bounded total reward objective on the induced finite MDP $\hat{\mdp}$,
which can be directly applied for objective $B_i$ on the subspace $K_i\setminus K_{i-2}$ in $\mdp$.
(The strategy $\zstrat_i$ is not necessarily unique, but our results
hold regardless of which of them is picked.)

We now define $\zstrat'$ by combining
different MD-strategies $\zstrat_i$, depending on the current state and on the
value of the 1-bit memory.
The intuition is that the strategy $\zstrat'$ has two modes: normal-mode
and next-mode.
In a state $\state' \in K_i \setminus K_{i-1}$, if the memory is $i\pmod 2$
then the strategy is in normal-mode and plays towards reaching $F_i$.
Otherwise, the strategy is in next-mode and plays towards reaching $F_{i+1}$
(normally this happens because $F_i$ has already been seen).

Initially $\zstrat'$ starts in a state $\state \in \initstates$ with the 1-bit memory set to
$1$.
We define the behavior of $\zstrat'$ in a state $\state' \in K_i \setminus K_{i-1}$
for every $i \ge 1$.
\begin{itemize}
\item
If the 1-bit memory is $i \pmod 2$ and $\state' \notin F_i$ then $\zstrat'$ plays like
$\zstrat_i$.
(Intuitively, one plays towards $F_i$, since one has not yet visited it.)
\item
If the 1-bit memory is $i \pmod 2$ and $\state' \in F_i$ then the 1-bit memory
is set to $(i+1) \pmod 2$, and $\zstrat'$ plays like $\zstrat_{i+1}$.
(Intuitively, one records the fact that one has already seen $F_i$ and then
targets the next set $F_{i+1}$.)
\item
If the 1-bit memory is $(i+1) \pmod 2$ then $\zstrat'$ plays like
$\zstrat_{i+1}$.
(Intuitively, one plays towards $F_{i+1}$, since one has already visited $F_i$.)
\end{itemize}
\begin{figure}[t]
   \begin{center}		
\centering

\makeatletter
\tikzset{
  use path for main/.code={%
    \tikz@addmode{%
      \expandafter\pgfsyssoftpath@setcurrentpath\csname tikz@intersect@path@name@#1\endcsname
    }%
  },
  use path for actions/.code={%
    \expandafter\def\expandafter\tikz@preactions\expandafter{\tikz@preactions\expandafter\let\expandafter\tikz@actions@path\csname tikz@intersect@path@name@#1\endcsname}%
  },
  use path/.style={%
    use path for main=#1,
    use path for actions=#1,
  }
}
\makeatother

\begin{tikzpicture}[>=latex',shorten >=1pt,node distance=1.9cm,on grid,auto,
fshade/.style={draw=none,fill=green!40!white,rotate=0},
bitone/.style={red,very thick},
bitzero/.style={blue, very thick},
lost/.style={densely dotted, ->, line width=1.5pt},
roundnode/.style={circle, draw,minimum size=1.5mm},
squarenode/.style={rectangle, draw,minimum size=2mm},
cross/.style={cross out,  fill=none, minimum size=2*(#1-\pgflinewidth), inner sep=0pt, outer sep=0pt}, cross/.default={1pt}]

\begin{scope}
\clip (5.8,0) ellipse (6.1cm and 2.3cm);
\draw[fshade] (-1,12) rectangle (12,0);
\end{scope}
\draw (5.8,0) ellipse (6.1cm and 2.3cm);

\draw [fill=white](4.2,0) ellipse (4.5cm and 1.9cm);

\begin{scope}
\clip (3.5,0) ellipse (3.8cm and 1.5cm);
\draw[fshade] (-1,10) rectangle (10,0);
\end{scope}
\draw (3.5,0) ellipse (3.8cm and 1.5cm);

\draw [fill=white](2.1,0) ellipse (2.4cm and 1.1cm);

\begin{scope}
\clip (1.2,0) ellipse (1.5cm and .7cm);
\draw[fshade] (-1,2) rectangle (3,0);
\end{scope}
\draw (1.2,0) ellipse (1.5cm and .7cm);


\node[draw=none](dot1) at (13,0)  {{\large $\cdots$}};
\node[draw=none](K1) at (2,-.3)   {$K_1$};
\node[draw=none](L1) at (4,-.3)   {$L_1$};
\node[draw=none](K2) at (6.8,-.3) {$K_2$};
\node[draw=none](L2) at (8.3,-.3) {$L_2$};
\node[draw=none](K3) at (11.5,-.3){$K_3$};

\coordinate (s) at (0,0);

\path[name path=firstrun] plot [smooth] coordinates {
    (s)
    (0.75,-.1)
    (1,0.8)
    (2.1,.75)
    (3,1.2)
    (3.8,.5)
    (6,2.1)
    (5.95,0.7)
    (8,2.5)
};
\begin{scope}
    \clip (s) rectangle (1,-0.5);
    \draw[bitzero,use path=firstrun];
\end{scope}
\begin{scope}
    \clip(s) rectangle (2.65,1.2);
    \draw[bitone,use path=firstrun];
\end{scope}
\begin{scope}
    \clip(2.65,-1.2) rectangle (5.54,2);
    \draw[bitzero,use path=firstrun];
\end{scope}
\begin{scope}
    \clip(5.54,2.5) rectangle (6.5,1.15);
    \draw[bitone,use path=firstrun];
\end{scope}
\begin{scope}
    \clip (5.9,1.15) -- (5,0) -- (7,0) -- (9,3) -- (8,3) --cycle;
    \draw[->,lost,use path=firstrun];
\end{scope}
\node[lost] at (8.25,2.75){$\pi_3$};

\draw[->,name path=secondrun,bitone] plot [smooth] coordinates {
    (s)
    (0.8,-.3)
    (2,.2)
    (3,-0.75)
    (4.5,0.6)
    (4.25,-0.25)
    (6,-2)
    (6,-0.75)
    (9,-1)
    (11,0.5)
    (7.7,.5)
    (10,2)
};
\begin{scope}
    \clip (0,0) -- (2,0) -- (1,-1.5) --cycle;
    \draw[bitzero,use path=secondrun];
\end{scope}
\begin{scope}
\clip(4.3,0.4) -- (4.25,1.5) -- (8,0) -- (10.9,0) -- (8,-2)-- (5,-2.1) -- (4,-0.5) --cycle;
    \draw[bitzero,use path=secondrun];
\end{scope}
\node[bitone] at (10.5,2.25){$\pi_2$};

\path [name path=thirdrun] plot [smooth] coordinates {
    (s)
    (0.7,-.5)
    (1.5,-0.5)
    (0.5,-2)
};
\begin{scope}
\clip (1.2,0) ellipse (1.5cm and .7cm);
    \draw[bitzero,use path=thirdrun];
\end{scope}
\begin{scope}
    \clip (0,-0.7) rectangle (2,-2);
    \draw[lost,->,use path=thirdrun];
\end{scope}
\node[lost] at (0.5,-2.25){$\pi_1$};

\node [roundnode,fill=white] at (s) {$I$};
\end{tikzpicture}
		\end{center}
                \caption{Memory updates along runs $\pi_1,\pi_2,\pi_3$,
                    drawn in blue while the memory-bit is one and in red while the bit is zero.
                    Both $\pi_1$ and $\pi_3$ violate $\varphi'$ and are drawn as dotted lines once they do.
}
\label{fig:flipingBit}
\end{figure}
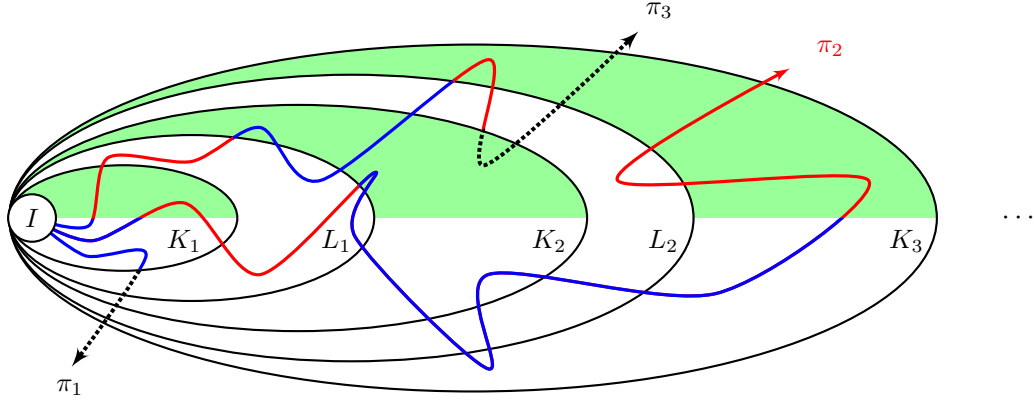

Observe that if a run according to $\zstrat'$ exits some set $K_i$
(and thus enters $K_{i+1}\setminus K_i$) with the bit still set to $i \pmod 2$
(normal-mode) then this run has not visited $F_i$ and thus does not satisfy the objective
$\formula'$. (Or the same has happened earlier for some $j < i$, in which case
also the objective $\formula'$ is violated.)
An example is the run $\pi_1$ in \cref{fig:flipingBit}.

However, if a run according to $\zstrat'$ exits some set $K_i$
(and thus enters $K_{i+1} \setminus K_i$) with the bit set to $(i+1) \pmod 2$
(thus $\zstrat_{i+1}$ in next-mode)
then in the new set $K_{i'} \setminus K_{i'-1}$ with $i'=i+1$ the bit is set to
$i' \pmod 2$ and $\zstrat'$ continues to play like $\zstrat_{i+1}$ in normal-mode.
Even if this run returns (temporarily) to $K_i$ (but not to $K_{i-1}$)
the strategy $\zstrat'$ continues to play like $\zstrat_{i+1}$ in next-mode.
An example is the run $\pi_2$ in \cref{fig:flipingBit}.

Finally, if a run returns to $K_{i-1}$ after having visited $F_i$
then it fails the objective $\formula'$.
An example is the run $\pi_3$ in \cref{fig:flipingBit}.

\medskip
\noindent
\smallparg{The 1-bit strategy $\zstrat'$ is optimal for $\formula'$ from every
$\state \in \initstates$.}
In the following let $\state \in \initstates$ be an arbitrary initial state in $\initstates$.
For any run from $\state$, let $\firstinset{F_i}$ be the first state
$\state'$ in $F_i$ that is visited (if any).
We define a bounded reward objective $B_i'$ for runs starting
at $\state$ as follows. Every run that does not satisfy the objective
$R_{\le i}$ gets assigned reward $0$.
Otherwise, consider a run from $\state$ that satisfies $R_{\le i}$.
When this run reaches the set $F_i$ for the first time
in some state $\state'$ then this run gets a reward of
$\valueof{\mdp,R_{\ge i+1}}{\state'}$. Note that this reward is $\le 1$.

We show that for all $i \in \N$
\begin{equation}\label{eq:Bi-prime}
\valueof{\mdp,\formula'}{\state} = \valueof{\mdp,B_i'}{\state}
\end{equation}
Towards the $\ge$ inequality,
let $\hat{\zstrat}$ be an $\hat{\eps}$-optimal strategy for
$B_i'$ from $\state$.
We define the strategy $\hat{\zstrat}'$ to play like $\hat{\zstrat}$
until a state $\state' \in F_i$ is reached and then to switch to
some $\hat{\eps}$-optimal strategy for objective $R_{\ge i+1}$
from $\state'$.
Every run from $\state$ that satisfies $\formula'$ can be split into parts,
before and after the first visit to the set $F_i$, i.e.,
$\formula' = \{w_1\state'w_2\ |\ w_1\state' \in R_{\le i}, \state' \in F_i,
\state'w_2 \in R_{\ge i+1}\}$.
Therefore we obtain that
$\probm_{\mdp,\state,\hat{\zstrat}'}(\formula') \ge
\expectval_{\mdp,\state,\hat{\zstrat}}(B_i') - \hat{\eps} \ge
\valueof{\mdp,B_i'}{\state} - 2\hat{\eps}$.
Since this holds for every
$\hat{\eps} >0$, we obtain
$\valueof{\mdp,\formula'}{\state} \ge \valueof{\mdp,B_i'}{\state}$.

Towards the $\le$ inequality,
let $\hat{\zstrat}$ be any strategy for
$\formula'$ from $\state$. We have
$\probm_{\mdp,\state,\hat{\zstrat}}(\formula')
\le
\sum_{\state' \in F_i}
\probm_{\mdp,\state,\hat{\zstrat}}(R_{\le i} \cap \firstinset{F_i}=\state')
\cdot \valueof{\mdp,R_{\ge i+1}}{\state'}
=
\expectval_{\mdp,\state,\hat{\zstrat}}(B_i')
$.
Thus $\valueof{\mdp,\formula'}{\state} \le \valueof{\mdp,B_i'}{\state}$.
Together we obtain \cref{eq:Bi-prime}.

For all $i \in \N$ and every state $\state' \in F_i$ we show that
\begin{equation}\label{eq:Ri-eq-Bi}
\valueof{\mdp,R_{\ge i+1}}{\state'} = \valueof{\mdp,B_{i+1}}{\state'}
\end{equation}
Towards the $\ge$ inequality,
let $\hat{\zstrat}$ be an $\hat{\eps}$-optimal strategy for
$B_{i+1}$ from $\state' \in F_i$.
We define the strategy $\hat{\zstrat}'$ to play like $\hat{\zstrat}$
until a state $\state'' \in F_{i+1}$ is reached and then to switch to
some $\hat{\eps}$-optimal strategy for objective $R_{\ge i+2}$
from $\state''$. We have that
$\probm_{\mdp,\state',\hat{\zstrat}'}(R_{\ge i+1}) \ge
\expectval_{\mdp,\state',\hat{\zstrat}}(B_{i+1}) - \hat{\eps} \ge
\valueof{\mdp,B_{i+1}}{\state} - 2\hat{\eps}$.
Since this holds for every $\hat{\eps} >0$, we obtain
$\valueof{\mdp,R_{\ge i+1}}{\state'} \ge \valueof{\mdp,B_{i+1}}{\state'}$.

Towards the $\le$ inequality,
let $\hat{\zstrat}$ be any strategy for
$R_{\ge i+1}$ from $\state' \in F_i$.
We have
\begin{align*}
\probm_{\mdp,\state',\hat{\zstrat}}(R_{\ge i+1})
~&\le
\sum_{\state'' \in F_{i+1}}
\probm_{\mdp,\state',\hat{\zstrat}}(R_{i+1}\states^\omega \cap \firstinset{F_{i+1}}=\state'')
\cdot \valueof{\mdp,R_{\ge i+2}}{\state''}\\
 &=
\expectval_{\mdp,\state',\hat{\zstrat}}(B_{i+1}).
\end{align*}
Thus $\valueof{\mdp,R_{\ge i+1}}{\state'} \le \valueof{\mdp,B_{i+1}}{\state'}$.
Together we obtain \cref{eq:Ri-eq-Bi}.

We show, by induction on $i$, that $\zstrat'$ is optimal for $B_i'$ for
all $i \in \N$ from start state $\state$, i.e.,
\begin{equation}\label{eq:opt-Bi-prime}
\expectval_{\mdp,\state,\zstrat'}(B_i') = \valueof{\mdp,B_i'}{\state}
\end{equation}
In the base case of $i=1$ we have that $B_1' = B_1$. The strategy $\zstrat'$ plays
$\zstrat_1$ until reaching $F_1$, which is optimal for objective $B_1$ and
thus optimal for $B_1'$.
For the induction step we assume (IH) that $\zstrat'$ is optimal for $B_i'$.
\begin{align*}
\valueof{\mdp,B_{i+1}'}{\state}\ & =\ \valueof{\mdp,B_i'}{\state} && \text{by \cref{eq:Bi-prime}}\\
                           & =\ \expectval_{\mdp,\state,\zstrat'}(B_i') && \text{by (IH)}\\
& =\ \sum_{\state' \in F_i}
\probm_{\mdp,\state,\zstrat'}(R_{\le i} \cap \firstinset{F_i}=\state')
\cdot \valueof{\mdp,R_{\ge i+1}}{\state'} && \text{by def.\ of $B_i'$}\\
& =\ \sum_{\state' \in F_i}
\probm_{\mdp,\state,\zstrat'}(R_{\le i} \cap \firstinset{F_i}=\state')
\cdot \valueof{\mdp,B_{i+1}}{\state'} && \text{by \cref{eq:Ri-eq-Bi}}\\
& =\ \sum_{\state' \in F_i}
\probm_{\mdp,\state,\zstrat'}(R_{\le i} \cap \firstinset{F_i}=\state')
\cdot \expectval_{\mdp,\state',\zstrat_{i+1}}(B_{i+1}) && \text{opt.\ of $\zstrat_{i+1}$ for $B_{i+1}$}\\
& =\ \expectval_{\mdp,\state,\zstrat'}(B_{i+1}') && \text{by def.\ of
  $\zstrat'$ and $B_{i+1}'$}
\end{align*}
So $\zstrat'$ attains the value $\valueof{\mdp,B_{i+1}'}{\state}$ of the
objective $B_{i+1}'$ from $\state$ and is optimal. Thus \cref{eq:opt-Bi-prime}.

Now we show that $\zstrat'$ performs well on the objectives $R_{\le i}$ for
all $i \in \N$.
\begin{equation}\label{eq:1-bit-val}
\probm_{\mdp,\state,\zstrat'}(R_{\le i}) \ge \valueof{\mdp,\formula'}{\state}
\end{equation}
We have
\begin{align*}
\probm_{\mdp,\state,\zstrat'}(R_{\le i})\ &
\ge\ \expectval_{\mdp,\state,\zstrat'}(B_i') && \text{since $B_i'$ gives rewards
  $0$ for runs $\notin R_{\le i}$ and $\le 1$ otherwise} \\
  & =\ \valueof{\mdp,B_i'}{\state} && \text{by \cref{eq:opt-Bi-prime}}\\
  & = \  \valueof{\mdp,\formula'}{\state} && \text{by \cref{eq:Bi-prime}}
\end{align*}
So we get \cref{eq:1-bit-val}. Now we are ready to prove the optimality of $\zstrat'$ for $\formula'$ from $\state$.
\begin{align*}
\probm_{\mdp,\state,\zstrat'}(\formula')\ &
  =\ \probm_{\mdp,\state,\zstrat'}(\cap_{i \in \N} R_{\le i}) && \text{by def.\ of $\formula'$}\\
  & = \  \lim_{i \to \infty}\probm_{\mdp,\state,\zstrat'}(R_{\le i}) &&
\text{by continuity of measures from above}\\
  & \ge \ \lim_{i \to \infty}\valueof{\mdp,\formula'}{\state} && \text{by \cref{eq:1-bit-val}}\\
  & = \ \valueof{\mdp,\formula'}{\state}
\end{align*}
This concludes the proof that $\zstrat'$ is optimal for $\formula'$ and hence
$2\eps$-optimal for $\formula$ for every initial state $\state \in \initstates$.

\medskip
\noindent
\smallparg{From finitely to infinitely branching MDPs.}
Let $\mdp$ be an infinitely branching acyclic MDP
with a finite set of initial states $\initstates$ and $\eps >0$. We derive a
finitely branching acyclic MDP $\mdp'$ with sufficiently similar behavior.
Every controlled state $x$ with infinite branching
$x \to y_i$ for all $i \in \N$ is replaced by a gadget
$x \to z_1, z_i \to z_{i+1}, z_i \to y_i$ for all $i \in \N$
with fresh controlled states $z_i$ (cf.\ \cref{fig:infmem-Buchi}).
Infinitely branching random states with $x \step{p_i}{} y_i$ for all $i \in \N$
are replaced by a gadget
$x \step{1}{} z_1, z_i \step{1-p_i'} z_{i+1}, z_i \step{p_i'} y_i$ for all
$i \in \N$, with fresh random states $z_i$ and suitably adjusted probabilities
$p_i'$ to ensure that the gadget
is left at state $y_i$ with probability $p_i$, i.e.,
$p_i' = p_i/(\prod_{j=1}^{i-1}(1-p_j'))$.

We apply the above result for
finitely branching acyclic MDPs to $\mdp'$
and obtain a 1-bit deterministic $\eps$-optimal strategy $\zstrat'$
for \Buchi\ from all states $\state \in \initstates$.
We construct a 1-bit deterministic $\eps$-optimal strategy $\zstrat''$
for $\mdp$ as follows.
Consider some state $x$ that is infinitely branching in $\mdp$
and its associated gadget in $\mdp'$.
Whenever a run in $\mdp'$ according to $\zstrat'$ reaches $x$ with some
memory value $\alpha \in \{0,1\}$ there exist
values $p_i$ for the probability that the gadget is left at state
$y_i$. Let $p \eqdef 1-\sum_{i\in \N} p_i$ be the probability that the gadget is
never left. (If $x$ is controlled then only one $p_i$ (or $p$) is nonzero,
since $\zstrat'$ is deterministic. If $x$ is random then $p=0$.)
Since $\zstrat'$ is deterministic, the memory updates are deterministic,
and thus there are values $\alpha_i' \in \{0,1\}$ such that whenever
the gadget is left at state $y_i$ the memory will be $\alpha_i'$.
We now define the behavior of the 1-bit deterministic strategy $\zstrat''$ at state $x$ with
memory $\alpha$ in $\mdp$.

If $x$ is controlled and $p\neq 1$ then $\zstrat''$ picks the successor state $y_i$ where
$p_i=1$ and sets the memory to $\alpha_i'$.
If $p=1$ then any run according to $\zstrat'$ that enters the gadget
does not satisfy the objective. Thus $\zstrat''$ performs at least as
well in $\mdp$ regardless of its choice, e.g., pick successor $y_1$ and
$\alpha' = \alpha$.

If $x$ is random then $p=0$ and the successor is chosen according to the defined
distribution (which is the same in $\mdp$ and $\mdp'$)
and $\zstrat''$ can only update its memory.
Whenever the successor $y_i$ is chosen, $\zstrat''$ updates the memory to
$\alpha_i'$.

In states that are not infinitely branching in $\mdp$, $\zstrat''$ does
exactly the same in $\mdp$ as $\zstrat'$ in $\mdp'$.

Since the gadgets do not intersect $F$, $\zstrat''$ performs
at least as well in $\mdp$ as $\zstrat'$ in $\mdp'$ and is thus
$\eps$-optimal from every $\state \in \initstates$.
\end{proof}

Now we show our upper bound on the strategy complexity of \Buchi\ for general
MDPs.

\thmMarkovplusone*
\begin{proof}
Let $\mdp=\mdptuple$ be a countable MDP with a finite set of initial
states $\initstates$ and $F \subseteq \states$ the
set of goal states.
We derive an acyclic MDP $\mdp' =
\tuple{\states',\zstates',\rstates',\transition',\probp'}$
by encoding a step-counter into the states.
Let $\states' = \states \times \N_0$, $\zstates' = \zstates \times \N_0$,
$\rstates' = \rstates \times \N_0$,
$\transition' = \{((x,n),(y,n+1))\ |\ (x,y) \in \transition\}$
and
$\probp'((x,n))((y,n+1)) = \probp(x)(y)$
for all $n \in \N_0$.
Let $\initstates' := \{(\state,0) \,|\, \state \in \initstates\}$
be the finite set of initial states of $\mdp'$ and $F' = F \times \N_0$ the set of goal states.

For every $\eps >0$, by \cref{thm:MDP-one-bit-Buchi}, there
exists a 1-bit deterministic strategy $\zstrat'$ for
$\reset(F)$ in $\mdp'$ that is $\eps$-optimal from every state $(\state,0) \in
\initstates'$.

We now define the deterministic 1-bit Markov strategy $\zstrat$
for $\reset(F)$ that is $\eps$-optimal from every $\state \in \initstates$ in $\mdp$.
It uses a step-counter (initially $0$) and $1$ extra bit of memory.
For any controlled state $\state'$, step-counter value $n$ and memory $\alpha
\in \{0,1\}$, consider the behavior of $\zstrat'$ at state $(\state',n)$
and memory $\alpha$. Let $(\state'',n+1)$ be the chosen successor state and
$\alpha'\in \{0,1\}$ the new memory content. Then $\zstrat$ chooses the successor
$\state''$, updates the memory to $\alpha'$ and increments the step-counter to
$n+1$.
\end{proof}
\end{document}